\documentclass[12pt]{amsart}
\usepackage{amsmath, amscd, amssymb, amsthm,amsfonts,mathrsfs,mathtools,bbm,graphicx}
\usepackage[nohug]{diagrams}
\usepackage[T1]{fontenc}
\pdfoutput=1
\makeatletter

\usepackage{color}
\usepackage{tikz}
\usetikzlibrary{matrix,arrows}

\theoremstyle{plain}
\newtheorem{theorem}{Theorem}[section]

\newtheorem{example}{Example}

\newtheorem*{dualityThm}{Theorem \ref{dualityIsomorphism}}
\newtheorem*{ringProp}{Proposition \ref{endRings}}

\newtheorem{proposition}[theorem]{Proposition}
\newtheorem{corollary}[theorem]{Corollary}
\newtheorem{lemma}[theorem]{Lemma}
\theoremstyle{remark}

\newtheorem*{remark}{Remark}
\theoremstyle{definition}

\newtheorem{definition}[theorem]{Definition}

\newtheorem*{notation}{Notation}

\numberwithin{equation}{section}

\renewcommand{\setminus}{\smallsetminus}
\newcommand{\inv}{^{-1}}
\newcommand{\R}{\mathbb{R}}
\newcommand{\Z}{\mathbb{Z}}
\newcommand{\N}{\mathbb{N}}

\newcommand{\C}{\mathbb{C}}

\newcommand{\e}{\varepsilon}
\renewcommand{\d}{\delta}
\renewcommand{\a}{\alpha}\renewcommand{\b}{\beta}

\renewcommand{\emptyset}{\varnothing}

\newcommand{\End}{\operatorname{End}}

\newcommand{\Id}{\mathbbm{1}}
\newcommand{\Hom}{\operatorname{Hom}}

\renewcommand{\matrix}[1]{\begin{bmatrix}#1\end{bmatrix}}

\renewcommand{\deg}{\operatorname{deg}}

\newcommand{\cotimes}{\:\bar\otimes\:}
\newcommand{\BN}[2]{{\mathcal{T}\hskip-2pt\mathcal{L}^{#1}_{#2}}}
\newcommand{\bn}[1]{\mathcal{T}\hskip-2pt\mathcal{L}_{#1}}
\newcommand{\Ch}{\operatorname{Ch}}

\newcommand{\Tot}{\operatorname{Tot}}

\newcommand{\Cob}{\operatorname{Cob}}

\newcommand{\TL}{\operatorname{TL}}
\newcommand{\AplusB}{\begin{tabular}{l l l}$A$\\ $\oplus$ \\ $B$\end{tabular}}

\newcommand{\Cone}{\operatorname{Cone}}

\newcommand{\ket}[1]{|{#1}\rangle}
\newcommand{\bra}[1]{\langle{#1}|}
\newcommand{\Tr}{\operatorname{Tr}}

\renewcommand{\sl}{\mathfrak{sl}}
\renewcommand{\L}[1]{L_{#1}}
\renewcommand{\R}[1]{R_{#1}}

\newarrow{Congruent}===={}

\newcommand{\pic}[1]{
\begin{minipage}{.2in}
\begin{center}\includegraphics[scale=.3]{#1}\end{center}
\end{minipage}
}

\newcommand{\mpic}[1]{
 \begin{minipage}{.14in}
\begin{center}\includegraphics[scale=.22]{#1}\end{center}
\end{minipage}
}
\newcommand{\bpic}[1]{
\begin{minipage}{.2in}
\begin{center}\includegraphics[scale=.4]{#1}\end{center}
\end{minipage}
}

\newcommand{\bmpic}[1]{\pic{#1}}

\newcommand{\turnback}{\pic{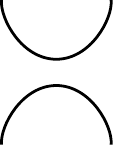}}
\newcommand{\straightthrough}{\pic{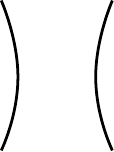}}
\newcommand{\topdot}{\mpic{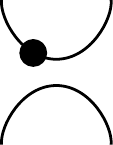}}
\newcommand{\bottomdot}{\mpic{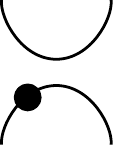}}
\newcommand{\differenceOfDots}{\topdot-\bottomdot\ }
\newcommand{\sumOfDots}{\topdot+\bottomdot\ }
\newcommand{\isaddle}{\mpic{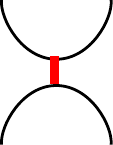}}

\newcommand{\bturnback}{\bpic{2strands_turnback.pdf}}
\newcommand{\bstraightthrough}{\bpic{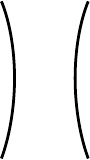}}
\newcommand{\btopdot}{\bmpic{topdot.pdf}}
\newcommand{\bbottomdot}{\bmpic{bottomdot.pdf}}
\newcommand{\bdifferenceOfDots}{\btopdot-\bbottomdot\ }
\newcommand{\bsumOfDots}{\btopdot+\bbottomdot\ }
\newcommand{\bisaddle}{\bmpic{isaddle.pdf}}

\begin{document}

\title{Morphisms between categorified spin networks}
\author{Matt Hogancamp}

\begin{abstract}
We introduce a graphical calculus for computing morphism spaces between the categorified spin networks of Cooper and Krushkal.  The calculus, phrased in terms of planar compositions of categorified Jones-Wenzl projectors and their duals, is then used to study the module structure of spin networks over the colored unknots.
\end{abstract}

\maketitle


\section{Introduction}
In the late 1990's Khovanov introduced \cite{Kh00} a homology theory for knots which categorifies the Jones polynomial.  Even though Khovanov's invariant is strictly stronger than the Jones polynomial, perhaps the single most important feature of Khovanov homology is the notion of morphisms between knot invariants.  Indeed, this new feature---more precisely functoriality up to sign under oriented link cobordisms---was exploited by Jacob Rasmussen in \cite{Ra10} to give a combinatorial proof of the Milnor conjecture, which had only previously been proven using gauge theory.  That the categorified knot invariant should detect subtle 4-dimensional information is exciting and points to the very reason that the idea of categorification was introduced to low-dimensional topology in \cite{CF94}.  Motivated by the desire to understand the additional structure present upon categorification, we are here concerned with computing spaces of morphisms between categorified spin networks.  On the pre-categorified level, spin networks are certain combinatorial objects which appear naturally in the study of the colored Jones polynomials---i.e.~$\sl_2$ quantum invariants, of which the Jones polynomial is a special case---and in the Turaev-Viro construction of quantum 3-manifold invariants \cite{KL94}.  

Specifically, a spin network $N$ is a trivalent graph with boundary, embedded in the disk, with edges labelled by nonnegative integers satisfying some admissibility conditions.  The evaluation of $N$ is a certain $\C(q)$-linear combination of tangles $T\subset D^2$ obtained by replacing an edge labelled by $n$ with a particular idempotent element $p_n\in \TL_n$ (called a Jones-Wenzl projector) of the Temperley-Lieb algebra and connecting up in the disk in a standard way:
\[
\text{A spin network: \ }\bpic{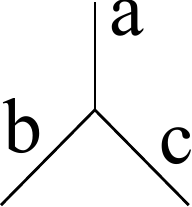} \hskip .6 in \text{Its evaluation: } \bpic{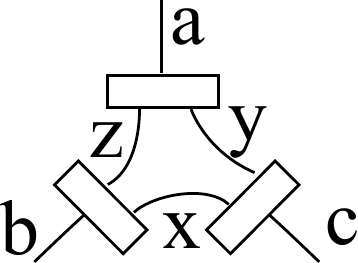}\ \ \ \ \ .
\]
Bar-Natan \cite{B-N05} defines graded cobordism categories $\bn{n}$ which categorify the algebras $\TL_n$, and Cooper-Krushkal \cite{CK12a} define chain complexes $P_n$ over $\bn{n}$ which categorify the Jones-Wenzl projectors.  The definition of (the evaluation of) spin networks extends immediately to the categorification.

We remark that the Temperley-Lieb algebras have many categorifications \cite{BFK99, El10, St05, Kh02, B-N05}, as do the Jones-Wenzl projectors \cite{FSS12, CK12a, Roz10a}.  By the uniqueness arguments in \cite{CK12a} and the conjectural relationship \cite{SS12} between the Lie-theoretic \cite{FSS12} and topological \cite{CK12a, Roz10a} categorifications of the Jones-Wenzl projectors, our choice to follow the Cooper-Krushkal construction is not very restrictive.  Also, we could work over categories of modules over Khovanov's rings $H^n$ \cite{Kh02} with only a slight change in notation.


The first step in computing spaces of morphisms between chain complexes over $\bn{n}$ is provided by the following theorem:
\begin{dualityThm}
There is a natural isomorphism
\begin{equation}\label{dualityEq}
\Hom^\bullet_{\bn{n}}(A,B)\cong q^{n}\Hom^\bullet_{\bn{0}^{\Pi}}\Big(\emptyset, \pic{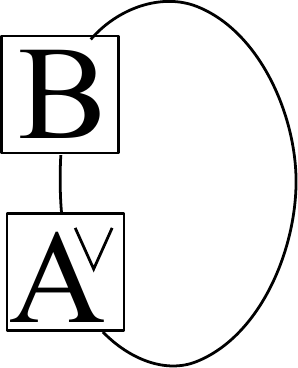}\ \ \  \Big).
\end{equation}
\end{dualityThm}
Here, $\Hom^\bullet$ denotes the chain complex spanned by homogeneous maps and differential given by the supercommutator $[d, -]$, and $(\ )^\vee$ is the contravariant duality functor which reflects diagrams in the plane and reverses all degrees.  Isomorphisms of this kind are common, in particular appearing in categories of $\sl_n$ foams \cite{MN08, Kh04, MSV09} and matrix factorizations \cite{KR08}.  The new feature here is the consideration of potentially unbounded chain complexes, hence the necessity to embed $\bn{n}$ into a category $\bn{n}^\Pi$ which contains countable direct products (see section \ref{completionSection}).

Diagrammatically, we let $\pic{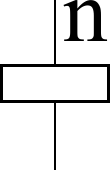}:= P_n$ denote the categorified Jones-Wenzl projector which is supported in non-positive homological degrees and $\pic{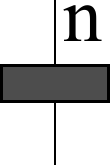}:=P_n^\vee$.  If $M$ and $N$ are planar compositions of projectors $P_n$ (for example, if $M$ and $N$ are categorified spin networks), Theorem $\ref{dualityIsomorphism}$ implies that $\Hom^\bullet(M,N)$ can be computed as follows:
\begin{enumerate}
\item Reflect $M$ and replace all the white boxes with black boxes to obtain $M^\vee$.
\item Glue up all of the loose ends of $N$ with the corresponding loose ends of $M^\vee$.
\item Take $\Hom^\bullet(\emptyset, -)$ of the result.  For example
\[
\Hom^\bullet\Big(\ \bpic{3vertex.pdf}\ \ \ \ \ \ \  \ \ , \bpic{3vertex.pdf}\ \ \ \ \ \ \ \ \Big) \cong q^{(a+b+c)/2}\Hom^\bullet\Big(\emptyset, \bpic{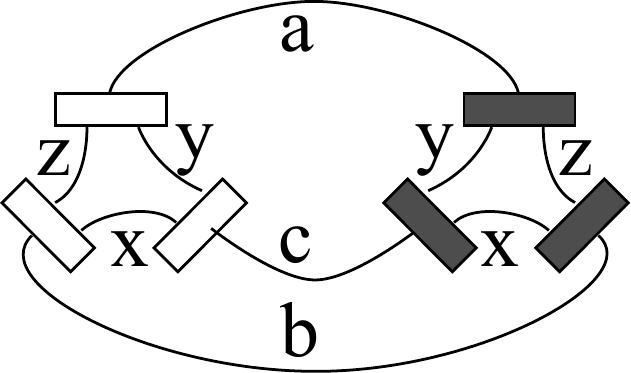}\ \ \ \ \ \ \ \ \ \  \ \ \ \ \ \Big).
\]
\end{enumerate}
We then prove that the resulting planar composition of $\pic{nProjector.pdf}$'s and $\pic{nProjectorDual.pdf}$'s can be simplified using the following relations
\begin{itemize}
\item[(4)] Diagrams which are isotopic rel boundary give canonically homotopy equivalent chain complexes (corollary \ref{isotopyCorollary}).
\item[(5)] Absorption rule: $\pic{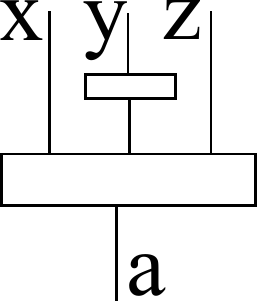}\ \ \ \simeq \bpic{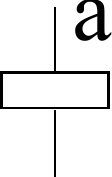} \simeq \pic{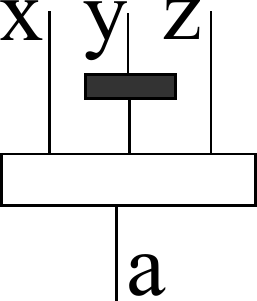}\ \ \ \ $ (proposition \ref{absorptionRule}).
\item[(6)] Commuting rule: $\pic{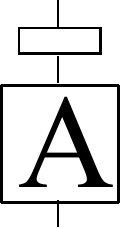}\simeq\pic{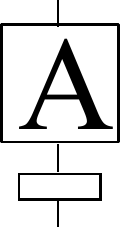}$ for every chain complex $A$ over $\bn{n}^\Pi$ (proposition \ref{commutingRule}).
\item[(7)] Semi-orthogonality rule: if $i<j$ then $\pic{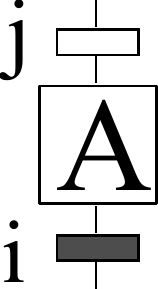}\simeq 0 $ for every chain complex $A$ over $\BN{i}{j}^\Pi$ (proposition \ref{semiRule}).
\end{itemize}
The apparent asymmetry in the above rules lies in the fact that here we are implicitly defining, for example, $\pic{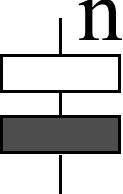}$ to be the total complex $\Tot^\Pi(A^{\bullet\bullet})$ of a bicomplex in which we take direct \emph{product} along the diagonals; there are dual statements in terms of direct \emph{sum}.  This graphical calculus enables us to simplify many $\Hom^\bullet$ spaces of interest.  In particular, we immediately obtain
\begin{ringProp}
Let $\ket{\ \ }$ denote $\Hom^\bullet_{\bn{0}}(\emptyset, \ \ )$.  We have
\begin{enumerate}
\item $\End^\bullet(\pic{aProjector} ) \simeq q^a \ket{\pic{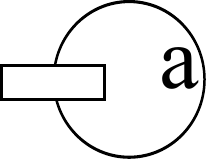}\ \  }$.
\item $\End^\bullet(\pic{3network}\ ) \simeq q^{(a+b+c)/2}\ket{\ \pic{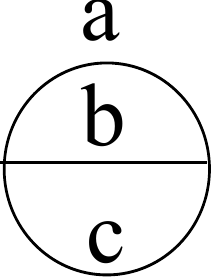}\ \ }$.
\item $\End^\bullet(\ \pic{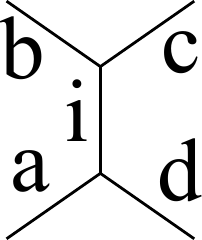}\ \ )\simeq q^{(a+b+c+d)/2} \ket{ \ \pic{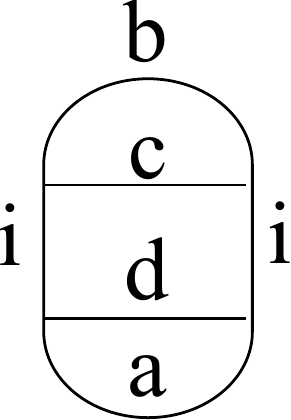}\ \ \ \ }$.
\end{enumerate}
\end{ringProp}
It follows in particular that the colored unknots give algebras which act on the edges of a categorified spin network.  In section \ref{sheetAlgSection} we show that these algebras are commutative, and the actions of these algebras on spin networks coincide with the topological actions, defined in terms of saddle cobordisms.

In a follow-up paper \cite{H12} we will use the $\Hom^\bullet$-space calculations to study functoriality properties of the categorification of spin networks and the associated link invariants.  Also, the calculus is used in the work \cite{CH12} of the author and Benjamin Cooper in which we categorify all of the minimal idempotents in $\TL_n$.

\textbf{Organization of the paper.}  In section \ref{precategorificationSection} we review the pre-categorified story, i.e.~Temperley-Lieb algebras and spin networks.  Section \ref{categorificationSection} recalls Bar-Natan's tangle categories and introduces their completions with respect to countable direct sums or products.  Section \ref{dualitySection} is dedicated to the statement and proof of the isomorphism in equation \ref{dualityEq}.  In section \ref{calculusSection} we develop the calculus which is used to simplify compositions of projectors and their duals, and section \ref{sheetAlgSection} contains the application to the action of colored unknots on spin networks.  Finally, the appendix (section \ref{homotopySection}) contains the necessary results in elementary homotopy theory, particularly on the contractibility of bicomplexes with contractible columns.

\textbf{Acknowledgements.}  The author would like to thank his advisor, Slava Krushkal, and Ben Cooper for many useful discussions and a lot of patience, particularly during early attempts to compute the homology of the sheet algebra $\End^\bullet(\pic{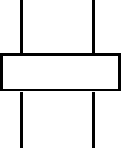})$.

\section{Pre-Categorification}\label{precategorificationSection}

\subsection{Temperley-Lieb algebras}
Let $\TL^m_n$ be the $\C(q)$-vector space generated by properly embedded 1-submanifolds of the rectangle $[0,1]^2$ with boundary equal to a standard set of $m$ points $\{i/(m+1)\:|\: 1\leq i\leq m\}$ on the ``top'' of the rectangle and $n$ points $\{i/(n+1)\:|\: 1\leq i\leq n\}$ on the ``bottom'' of the rectangle.  We regard diagrams modulo planar isotopy and the relation $D\sqcup U = (q+q\inv)D$, where $U$ is a circle component.

We have a pairing $\TL^m_k\otimes \TL^k_n\rightarrow \TL^m_n$ given by vertical stacking, which we denote by $a\cdot b$, or simply $ab$.  The pairing makes $\TL_n:=\TL^n_n$ into an algebra, called the \emph{Temperley-Lieb algebra} on $n$ strands, and $\TL^m_n$ into a $(\TL_m,\TL_n)$-bimodule.  For a diagram $a\in \TL^m_n$, define the \emph{through degree} $\tau(a)$ to be the minimal $k$ such that $a=b\cdot c$ with $b\in\TL^m_k$, $c\in\TL^k_n$.  For a linear combinations $b=\sum_a f_a a$ of diagrams, let $\tau(b):=\max\{\tau(a)\:|\: f_a\neq 0\}$

The Temperley-Lieb algebras are very closely linked with the representation theory of quantum $\sl_2$; generally $\TL^m_n$ is the vector space in which the Jones invariant (i.e.\ $\sl_2$ quantum invariant) of tangles lives.

\subsection{The Jones-Wenzl projectors}
The connection of the algebra $\TL_n$ to representation theory of quantum $\sl_2$ comes from the fact that
\[
\TL_n\cong \End_{U_q(\sl_2)}(V^{\otimes n})
\]
where $U_q(\sl_2)$ is a $q$-deformed version of the enveloping algebra of $\sl_2$, and $V$ is the $q$-deformed version of the standard 2-dimensional representation.  The idempotent $p_n\in \TL_n$ corresponding to projection onto the $n$-th symmetric power is called the \emph{Jones-Wenzl projector}, and is essential in defining the $\sl_2$ quantum invariants for links and 3-manifolds.  This idempotent can be defined axiomatically as in the following theorem.

\begin{theorem}\label{JWproj}
There are elements $p_n\in \TL_n$ uniquely characterized by (1) $a\cdot p_n = p_n\cdot b = 0$ whenever $\tau(a),\tau(b)<n$ and (2) $\tau(p_n- 1_n)<n$.
\end{theorem}

We refer to property (1) of the theorem by saying that $p_n$ \emph{kills turnbacks}.  Indeed, using the graphical notation in which we denote $a$ parallel strands by $\pic{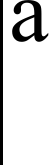}$ and $p_n:=\pic{nProjector}$, property (1) becomes equivalent to
\[
\pic{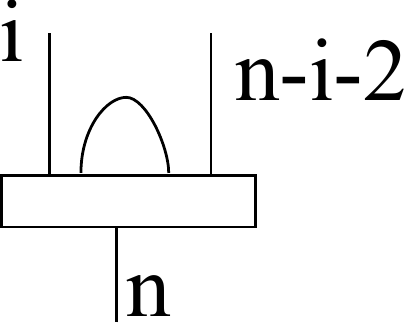}\ \ \ \ =\pic{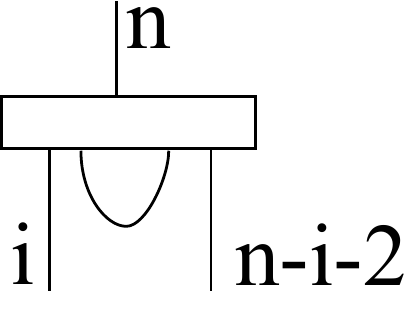} \ \ \ \ =0
\]
for $0\leq i\leq n-2$.

\section{The categorification}\label{categorificationSection}
In this section we describe the categorification of the spaces $\TL^m_n$ due to Bar-Natan \cite{B-N05} and the categorification of the idempotent $p_n\in \TL_n$ due to Cooper-Krushkal.  The definition of the evaluation of spin networks lifts immediately to the categorification, and so we can speak of categorified (evaluations of) spin networks.

\subsection{Some basic categorical notions}
A category $\mathscr{A}$ is said to be $\Z$-linear if the morphism spaces are abelian groups, and composition induces a linear map $\Hom_{\mathscr{A}}(b,c)\otimes \Hom_{\mathscr{A}}(a,b)\rightarrow \Hom_{\mathscr{A}}(a,c)$.  If, in addition, $\mathscr{A}$ contains finite direct sums, then we say that $\mathscr{A}$ is \emph{additive}.  A functor $F:\mathscr{A}\rightarrow \mathscr{B}$ between $\Z$-linear categories is said to be $\Z$-linear if it induces a linear map $F:\Hom_\mathscr{A}(A,A')\rightarrow \Hom_\mathscr{B}(F(A),F(A'))$.  A functor $F:\mathscr{A}_1\times \dots \times \mathscr{A}_r\rightarrow \mathscr{B}$ is said to be \emph{multilinear} if it induces a map
\[ 
F:\Hom_{\mathscr{A}_1}(A_1,B_1)\otimes \dots \otimes \Hom_{\mathscr{A}_r}(A_r,B_r)\rightarrow \Hom_{\mathscr{B}}(F(A_1,\ldots, A_r), F(B_1,\ldots, B_r)).
\]

Let $\mathscr{A}^\Z$ denote the graded category over $\mathscr{A}$, i.e.\ objects are sequences $A^\bullet = (A^k)_{k\in \Z}$ with $A^k\in\mathscr{A}$, and morphisms are degree preserving multimaps.
\begin{definition}
For graded objects $A,B\in \mathscr{A}^\Z$, let $\Hom^\bullet_\mathscr{A}(A,B)$ be the graded abelian group spanned by homogeneous maps $A\rightarrow B$.  If $A$ and $B$ are chain complexes, then so is $\Hom^\bullet_\mathscr{A}(A,B)$, with differential given by the super-commutator: for $f\in \Hom^k_\mathscr{A}(A,B)$, define $D(f)$ to be
\[
[d,f] := d_B\circ f - (-1)^k f\circ d_A
\]
\end{definition}
The degree zero cycles are precisely the chain maps in the usual sense, and the homology group $H^k(\Hom^\bullet_\mathscr{A}(A,B))$ is equal to the space of degree zero chain maps $t^{-k}A\rightarrow B$ modulo chain homotopy.  Here $tA^\bullet$ denotes the chain complex $(tA)^k = A^{k-1}$ with differential $-d_A$.  We typically denote the homological degree of a morphism by $|f|=k$ if $f\in \Hom^k_\mathscr{A}(A,B)$.

Let $\Ch(\mathscr{A})$ be the category of chain complexes over $\mathscr{A}$ with differentials of degree $+1$ and morphism spaces $\Hom^\bullet_\mathscr{A}(\ ,\ )$.  We also have the full subcategories $\Ch^{\leq 0}(\mathscr{A})$, $\Ch^{\geq 0}(\mathscr{A})$, $\Ch^b(\mathscr{A})$, consisting of chain complexes which are concentrated in non-negative homological degrees, concentrated in non-positive degrees, respectively bounded.

\subsection{Categorification of Temperley-Lieb}
In \cite{B-N05} Bar-Natan regards generators of the space $\TL^m_n$ as objects of a (graded, additive) category $\BN{m}{n}$ in which the morphisms ensure that the defining relations for $\TL^m_n$ lift to isomorphisms.  Gluing of diagrams in the plane gives maps $\TL_{n_1}\otimes \dots \otimes \TL_{n_r}\rightarrow \TL_{n_0}$, and this lifts to mulilinear functors $\bn{n_1}\times \dots \times \bn{n_r}\rightarrow \bn{n_0}$.  In \cite{CK12b} Cooper and Krushkal prove that the split Grothendieck group $K_0(\bn{n})$ (actually a $\Z[q,q\inv]$-algebra) satisfies
\[
\C(q)\otimes_{\Z[q,q\inv]} K_0(\bn{n})\cong \TL_n,
\]
so we say that $\TL_n$ is categorified by $\bn{n}$.

\begin{definition}
Fix an identification $D^2\cong [0,1]^2$, and let $\Cob^m_n$ be the category with
\begin{itemize}
\item Objects of $\Cob^m_n$: properly embedded 1-submanifolds $T\subset D^2$ with $\partial T = \{\frac{i}{m+1}\:|\: 1\leq i\leq m\}\times\{1\}\cup \{\frac{j}{n+1}\:|\: 1\leq j\leq n\}\times\{0\}$.
\item Morphisms: $\Z[\a]$ linear combinations of (nicely embedded) cobordisms in $D^2\times [0,1]$ decorated with dots, modulo isotopy which fixes the boundary, and the following local relations:
\vskip 4pt
	\begin{enumerate}
		\item $\bpic{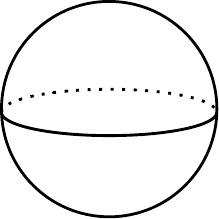}\ \ \ =0$, $\bpic{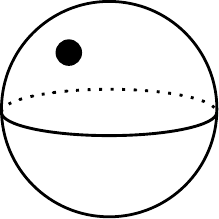}\ \ \  = 1$, $\bpic{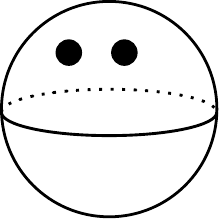}\ \ \ =0$, and $\bpic{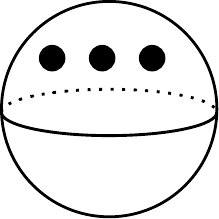}\ \ \ =\a$
		\item $\bpic{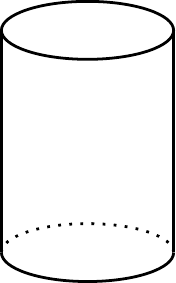}\ \ \ =\bpic{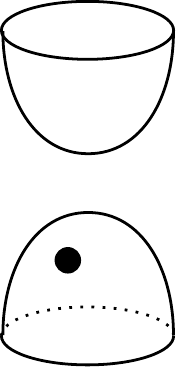}\ \ +\ \bpic{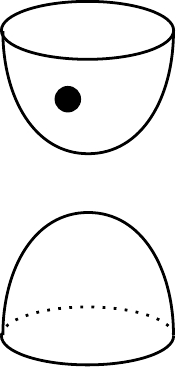}\ \ $
		\item $\bpic{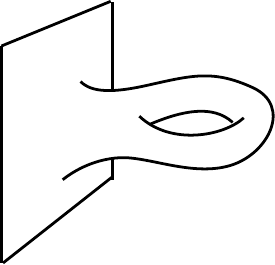}\ \ \ \ \ \ = 2\:\bpic{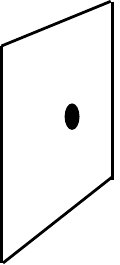}\ $.
	\end{enumerate}
\end{itemize}
Composition of morphisms in $\Cob^m_n$ is induced by stacking.
\end{definition}

Here $\a$ is a formal parameter; specializing to $\a=0$ gives the usual setting for Khovanov homology while specializing to $\a=1$ (and inverting 2) gives Lee's degeneration of Khovanov homology.  We make $\Cob^m_n$ into a graded category by introducing formal degree shifts $q^k a$ of objects, and defining the degree of a morphism $f:q^k a\rightarrow q^l b$ by setting $\deg \a =4$ and defining the degree of a cobordism $S$ by
\[
\deg(S)=l-k-\chi(S)+(m+n)/2
\]
Additivity of Euler characteristic under composition of cobordisms ensures that $\Cob^m_n$ becomes a graded category.

\begin{notation}
In what follows we work with differential graded categories of chain complexes over the graded categories $\Cob^m_n$, and so our moprhism spaces are actually bigraded.  We will refer to the \emph{homological degree} $\deg_h(f)$, the \emph{quantum degree} or $q$-degree $\deg_q(f)$ and the \emph{bidegree} $\deg(f) = (\deg_h(f),\deg_q(f))$.
\end{notation}

\subsection{Formally adjoining direct sums and products}\label{completionSection}
The categories $\Cob^n_m$ are not additive, meaning they do not contain all finite direct sums.  This makes it impossible to use many standard constructions in homological algebra, for example taking total complexes of a bicomplex.  We rectify the situation by formally adjoining sums (or products) to the categories $\Cob^m_n$.

\begin{definition}
Let $\BN{m}{n}$, $(\BN{m}{n})^\oplus$, $(\BN{m}{n})^\Pi$ be the closures of $\Cob^m_n$ under finite direct sums (equivalently products), countable direct sums, respectively countable direct products.  That is to say,
\begin{enumerate}
\item Let $\BN{m}{n}$ be the category with objects symbols $\bigoplus_{i=1}^k a_i $ with $a_i\in \Cob^m_n$ for $i\in\{1,\ldots,k\}$, and morphisms $\bigoplus_{i=1}^k a_i\rightarrow \bigoplus_{j=1}^l b_j$ given by matrices $({}_jf_i:a_i\rightarrow b_j)_{i,j}$.
\item Let $(\BN{m}{n})^\oplus$ be the category with objects the symbols $\bigoplus_{i\geq 1}a_i$ with $a_i\in \Cob^m_n$, $i\in\{1,2,\ldots\}$, and morphisms $\bigoplus_{i\geq 1}a_i\rightarrow \bigoplus_{j\geq 1}b_j$ given by matrices $({}_j f_i)\in \prod_{i,j\geq 1}\Hom_{\mathscr{A}}(a_i,b_j)$ with finite columns, i.e.\ for fixed $i$, ${}_j f_i=0$ for all but finitely many $j$.
\item Let $(\BN{m}{n})^\Pi$ be the category with objects the symbols $\prod_{i\geq 1}a_i$ with $a_i\in \Cob^m_n$, $i\in\{1,2,\ldots\}$, and morphisms $\prod_{i\geq 1}a_i\rightarrow \prod_{j\geq 1}b_j$ given by matrices $({}_j f_i)_\in \prod_{i,j\geq 1}\Hom_{\mathscr{A}}(a_i,b_j)$ with finite rows, i.e.\ for fixed $j$, ${}_j f_i=0$ for all but finitely many $i$.
\end{enumerate}
In any case composition of morphisms is given by matrix multiplication: ${}_k (f\circ g)_i = \sum_j {}_k f_j\circ {}_j g_i$ which is always a finite sum, in light of the finiteness conditions on morphisms.  In order for the categories $\BN{m}{n}^\oplus$ and $\BN{m}{n}^\Pi$ to be graded, it is necessary to assume additionally that any morphism is a finite sum of homogeneous morphisms, i.e. morphisms $({}_j f_i)$ with $\deg_q({}_j f_i) = \deg_q({}_{j'}f_{i'})$ for all $i,i',j,j'$.
\end{definition}

\begin{proposition}
$\BN{m}{n}$ contains finite direct sums, and contains $\Cob^m_n$ as a full subcategory.  $(\BN{m}{n})^\oplus$ and $(\BN{m}{n})^\Pi$ contain countable direct sums, respectively countable direct products, and each contains $\BN{m}{n}$ as a full subcategory.
\end{proposition}
\begin{proof}
Straightforward.
\end{proof}

\subsection{Interpreting pictures as chain complexes}
Objects of $\BN{m}{n}$ can be glued together in the plane in precisely the same way as elements of $\TL^m_n$ (extending by multilinearity to direct sums of diagrams).  But more is true, we can also glue morphisms (linear combinations of surfaces in $[0,1]^3$) in the same way.  Collections of categories with this sort of algebraic structure are called canopolies in \cite{B-N05} and canopolises in \cite{MN08}.

We can also compose chain complexes together in the plane.  For an excellent introduction, see \cite{B-N05}.  The only differences here are that (1) we consider potentially unbounded chain complexes, hence the necessity for our categories to contain infinite direct sums or products, and (2) we keep track of partition of the boundary into ``top'' and ``bottom.''  We will describe the basic idea with an example.  Fix $A\in\Ch((\BN{4}{2})^\oplus)$, $B\in\Ch((\BN{2}{2})^\oplus)$, and let $T(A,B)$ be the picture
\[
T(A,B) = \pic{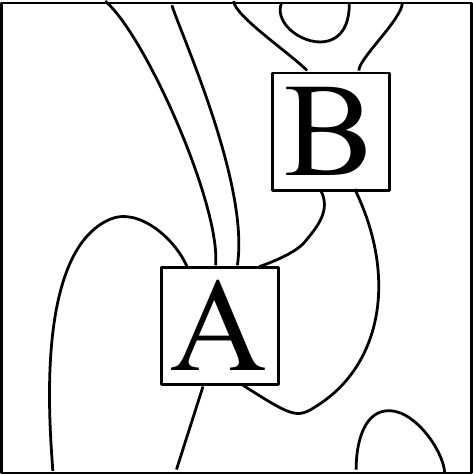}
\]
We will describe in steps how to interpret $T(A,B)$ as a chain complex over $(\BN{6}{4})^\oplus$ in a functorial way.  

\begin{enumerate}
\item If $A=a$ and $B=b$ are objects of the appropriate $\Cob^m_n$, then because of the topological nature of the categories $\Cob^m_n$, we may define $T(a,b)$ to be the object given by gluing diagrams together.  Similarly, if $f:a\rightarrow a'$, $g:b\rightarrow b'$ are morphisms of $q$-degree zero, then we can define a morphism $T(f,g):T(a,b)\rightarrow T(a',b')$ by gluing cobordisms $f$, $g$, together with the identity cobordism away from $a,b$.
\item  We can extend $T(\ ,\ )$ linearly in each argument, obtaining a bilinear functor $T(\ ,\ ): (\BN{4}{2})^\oplus\times (\BN{2}{2})^\oplus\rightarrow (\BN{6}{4})^\oplus$.
\item  Suppose $(A^\bullet, d_A)\in \Ch(\BN{4}{2}^\oplus)$, $(B^\bullet, d_B)\in\Ch(\BN{2}{2}^\oplus)$ are arbitrary.  Define $T(A,B)^\bullet$ to be the chain complex $T(A,B)^k = \bigoplus_{i+j=k}T(A^i,B^j)$ with differential $d_{T(A,B)} = T(d_A,\Id_B)+T(\Id_A,d_B)$, where the action on morphisms is given by the Koszul sign rule: $T(f,g)|_{T(A^i,B^j)} = (-1)^{i|g|}T(f|_{A^i}, g|_{B^j})$.
\end{enumerate}

More generally, if $T(A_1,\ldots,A_r)$ is a similar looking picture with $r$ ``inputs,'' and $A_i$ are chain complexes over the appropriate categories $(\BN{m}{n})^\oplus$, then $T(A_1,\ldots, A_r)$ can be interpreted as the total complex of a multicomplex, which we refer to as a \emph{planar composition} of the $A_i$.  Reordering of the arguments gives a different total differential, but the resulting chain complexes are isomorphic.  In fact  this isomorphism is natural in the $A_j^\bullet$.  This proof is familiar to anyone who has shown that Khovanov homology does not depend on the ordering of crossings in a knot diagram.

\begin{remark}
We can compose chain complexes over the categories $(\BN{m}{n})^\Pi$ together in the plane in precisely the same way, replacing $\oplus$ everywhere with $\Pi$.  There is some ambiguity, however in defining the chain complex $T(A_1,\ldots, A_r)$ when $A_i$ are potentially unbounded chain complexes over $\BN{m_i}{n_i}$: we can take total complex using direct sum and regard the result in $\Ch(\BN{m}{n}^\oplus)$, or we can take total complex using direct product and regard the result in $\Ch(\BN{m}{n}^\Pi)$.  We resolve the ambiguity by including the symbol $\oplus$ or $\Pi$ somewhere in our pictures.  When there is no ambiguity, for example if $A_i\in\Ch^{\leq 0}(\BN{m_i}{n_i})$, then will often omit the symbol and regard the result as a chain complex over the appropriate $\Ch^{\leq 0}(\BN{m}{n})$.  Also, if $r=1$ then there is no ambiguity in defining $T(A)$ for $A\in \Ch(\BN{m}{n})$.
\end{remark}

We give special notation to certain planar compositions.
\begin{definition}\label{planarOpsDef} Define the following covariant functors:
\begin{itemize}
\item Let $\otimes:\Cob^m_k\times \Cob^k_n\rightarrow \Cob^m_n$ be $a \otimes b := \bpic{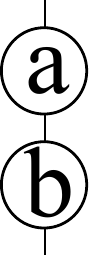}$.
\item Let $\Tr:\Cob_{n}\rightarrow \Cob_{0}$ be $\Tr(a) = \bpic{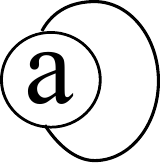}\ \ \ $.
\item Let $r:\Cob^{m}_{n}\rightarrow \Cob^{n}_{m}$ be $r(a) := \bpic{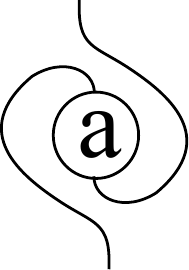}\ \ \ $.
\end{itemize}
Denote similarly the extensions to categories of chain complexes over the appropriate $\BN{m}{n}$ and $(\BN{m}{n})^\oplus$.  Denote by the $\cotimes$ the extension of $\otimes$ to the categories of chain complexes over the appropriate $(\BN{m}{n})^\Pi$.
\end{definition}

\section{Duality}\label{dualitySection}

This section is dedicated to the proof that
\begin{equation}\label{dualityEq2}
\Hom^\bullet_{\BN{n}{m}}(A,B)\cong q^{(m+n)/2}\ket{\Tr(B\cotimes A^\vee)}
\end{equation}
naturally.  Here $(\ \ )^\vee$ is the contravariant functor which reflects all diagrams across the $x$-axis and reverses all degrees, $\ket{C}:=\Hom^\bullet_{\bn{0}^\Pi}(\emptyset, C)$, and $\Tr$ is the markov trace.

That equation \ref{dualityEq2} holds on the level of bigraded abelian groups follows from simple facts about the cobordism categories $\Cob^n_m$.  That the isomorphism respects the differentials will follow from naturality, which requires a little book-keeping.  To even have a well defined notion of naturality, we must first describe how $(\ \ )^\vee$ and $\Hom^\bullet$ are functorial in a differential graded sense.  The extension of (multi-) linear functors to differential graded functors on categories of chain complexes is fairly straightforward, but rather than appeal to general theory we describe the action of $(\ \ )^\vee$ and $\Hom^\bullet$ on morphisms explicitly.

\subsection{The functor $(\ \ )^\vee$}
\begin{definition}\label{dualityFunctor}
Let $(\ )^\vee: \BN{n}{m}\leftrightarrow \BN{m}{n}$ be the contravariant \emph{duality functor} which acts on objects by reflecting all diagrams about the $x$-axis and reversing $q$-degree, and which acts on morphisms (matrices of cobordisms in $[0,1]\times [0,1]\times[0,1]$) by applying $(x,y,z)\mapsto (x,1-y,1-z)$ to each entry and taking transpose of the result.
\end{definition}

We have an extension of $(\ )^\vee$ to mutually inverse contravariant functors $\BN{n}{m}^\oplus\leftrightarrow \BN{m}{n}^\Pi$ defined on objects by $(\bigoplus_i a_i)^\vee = \prod_i a_i^\vee$ and $(\prod_i a_i)^\vee = \bigoplus_i a_i^\vee$.

We would like to extend the functor $(\ )^\vee$ to chain complexes.  I.e.~we want a chain complex $A^\vee$ for each $A\in\Ch(\BN{n}{m})$ and a chain map $\Hom^\bullet(A,B)\rightarrow \Hom^\bullet(B^\vee,A^\vee)$ which is compatible with composition of morphisms in the (contravariant) differential graded sense.

\begin{definition}
For $A\in \Ch(\BN{n}{m}^\oplus)$, define $A^\vee\in \Ch(\BN{m}{n}^\Pi)$ to be the chain complex $(A^\vee)^k = (A^{-k})^\vee$ with differential $-d(_A)^\vee$, where $d_A$ is the differential on $A$.  On morphisms, for $f\in\Hom^k_{\BN{n}{m}^\oplus}(A,B)$ we define $f^\vee\in\Hom^k_{\BN{m}{n}^\Pi}(B^\vee,A^\vee)$ by commutativity of the following square
\[
\begin{diagram}
(B^\vee)^i & \rTo^{(f^\vee)_i}& (A^\vee)_{i+k}\\
\dTo^{=} & & \dTo_{=}\\
(B^{-i})^\vee & \rTo^{(-1)^{ik}f_{-i-k}^\vee} & A_{-i-k}.
\end{diagram}
\]
We have similar definitions with the roles of $\Pi$ and $\oplus$ reversed.
\end{definition}

\begin{proposition}
We have
\begin{enumerate}
\item $(\ )^\vee$ induces a degree zero chain map $\Hom^\bullet_{\BN{n}{m}^\oplus}(A,B) \rightarrow \Hom^\bullet_{\BN{m}{n}^\Pi}(B^\vee,A^\vee)$.
\item $(f\circ g)^\vee = (-1)^{|f||g|}g^\vee\circ f^\vee$.\qed
\end{enumerate}
\end{proposition}

The duality functor interacts nicely with the planar compositions of chain complexes:

\begin{proposition}\label{dualityAndPlanarComp}
If $T^\oplus(A_1,\ldots, A_r)$ is a planar composition of complexes $A_i\in\Ch((\BN{m_i}{n_i})^\oplus)$, then $T^\oplus(A_1,\ldots, A_r)^\vee\cong \bar T^\Pi(A_1^\vee,\ldots, A_r^\vee)$ naturally, where $\bar T$ is the picture obtained from $T$ by vertical reflection.  In particular $(A \otimes  B)^\vee$ is naturally isomorphic to  $B^\vee \cotimes A^\vee$.\qed
\end{proposition}

\subsection{The functor $\Hom^\bullet_\mathscr{A}$}

\begin{definition}\label{homAsFunctor}
Let $\mathscr{A}$ be a $\Z$-linear category, and let $A,B,C$ be graded objects over $\mathscr{A}$.  Define degree preserving linear maps
\begin{itemize}
\item For $f\in \Hom^k_\mathscr{A}(A,B)$, define $R_f:\Hom^\bullet_\mathscr{A}(B,C)\rightarrow \Hom^\bullet_\mathscr{A}(A,C)$ by $R_f(\a) = (-1)^{|\a||f|}\a\circ f$.
\item For $g\in \Hom^l_\mathscr{A}(B,C)$, define $L_g:\Hom^\bullet_\mathscr{A}(A,B)\rightarrow \Hom^\bullet_\mathscr{A}(A,C)$ by $L_g(\b) = g\circ b$.
\end{itemize}
Put $\Hom^\bullet_{\mathscr{A}}(f\otimes g) := R_f\circ L_g$.
\end{definition}

The proof of the following is straightforward.

\begin{proposition}\label{LRprop}
We have
\begin{enumerate}
\item $L_f\circ L_{f'} = L_{f\circ f'}$, $R_g\circ R_{g'} = (-1)^{|g||g'|}R_{g'\circ g}$, and $R_f\circ L_g = (-1)^{|f||g|}L_g\circ R_f$.
\item $\Hom^\bullet_\mathscr{A}$ gives a degree zero chain map
\[\Hom^\bullet_\mathscr{A}(A',A)\otimes \Hom^\bullet_\mathscr{A}(B,B')\rightarrow \Hom^\bullet_{\Z}(\Hom^\bullet_\mathscr{A}(A,B),\Hom^\bullet_\mathscr{A}(A',B'))
\]
\item The differential on $\Hom^\bullet_\mathscr{A}(A,B)$ is $L_{d_B}- R_{d_A}$. \qed
\end{enumerate}.
\end{proposition}

\subsection{Ordinary duality} \label{ordinaryDualitySection}

\begin{definition}
Let $\ket{\ \ }:\bn{0}\rightarrow \Z-\mathbf{mod}$ be the functor $\ket{c} = \Hom(\emptyset, c)$.
\end{definition}

When $A^\bullet = A^0 =:a\in \BN{0}{2n}$ and $B^\bullet = B^0 =: b\in\BN{0}{2n}$, equation \ref{dualityEq2} becomes

\begin{proposition}\label{ordinaryDualityIso}
$\Hom_{\BN{0}{2n}} (a,b) \cong q^n\ket{b\otimes a^\vee}$ naturally.
\end{proposition}
Before proving this, we introduce some notation.

\begin{definition}
Let $a\in\BN{0}{2n}$ be indecomposable.  That is to say, $a$ is a diagram with no circle components.
\begin{itemize}
\item Let $\eta_a:\emptyset\rightarrow a\otimes a^\vee$ be the map of $q$-degree $-n$ given by the minimal cobordism.  More precisely, $a\otimes a^\vee$ consists of $n$ disjoint circles, and $\eta_a$ is $n$ disjoint disks which cap off each of these components.
\item Let $s_a:a^\vee \otimes a\rightarrow 1$ be the map of $q$-degree $n$ given by the minimal cobordism.  More precisely, $a^\vee\otimes a$ is the disjoint union of $a$ and its reflection, and $s_a$ is the cobordism given by attaching 1-handles on matching pairs of components ($n$ in total).
\end{itemize}
\end{definition}

\begin{remark}
The idea of the proof of proposition \ref{ordinaryDualityIso} can be represented schematically as follows.  Indicate objects of $\BN{0}{2n}$ by arcs, and denote $f\in\Hom_{\BN{0}{2n}}(a,b)$, $\zeta\in\Hom_{\bn{0}}(\emptyset, b\otimes a^\vee)$, $\eta_a$, and $s_a$ as in (here cobordisms are read top-to-bottom):
\[
f = \pic{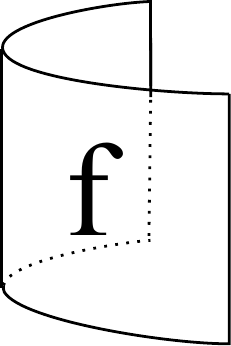}\ \ \ \ \ \ \ \ \ \zeta = \pic{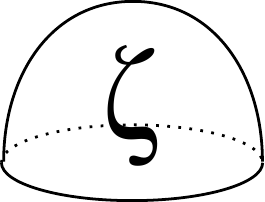} \ \ \ \ \ \ \ \ \ \eta_a = \pic{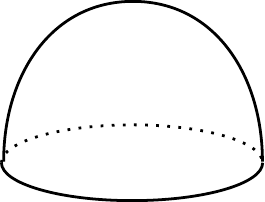}\ \ \ \ \ \ \ \  \ s_a = \pic{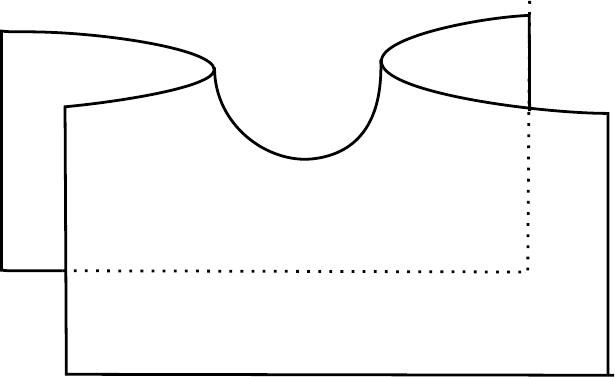}
\]
In proposition \ref{unitAndSaddleProp} we will prove that
\[
\pic{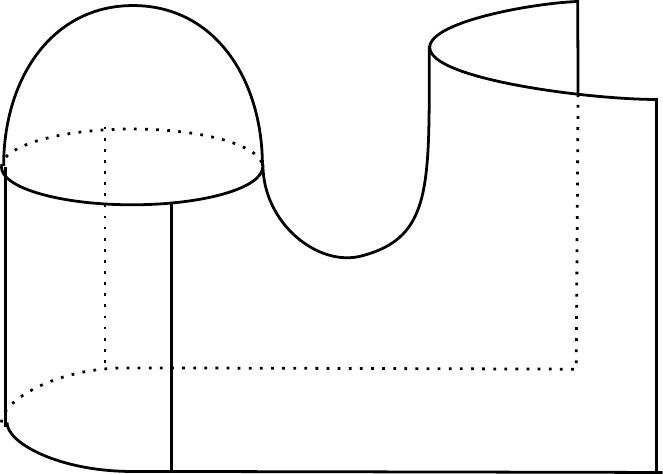} \ \ \ \  \ \ \ \ \ \ \ \ = \pic{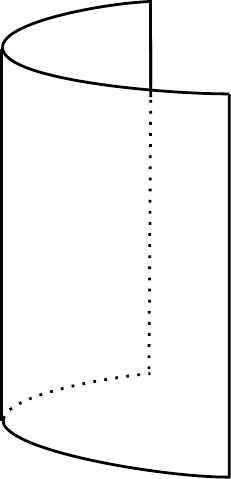} \ \ \ \ , \hskip .2in  \pic{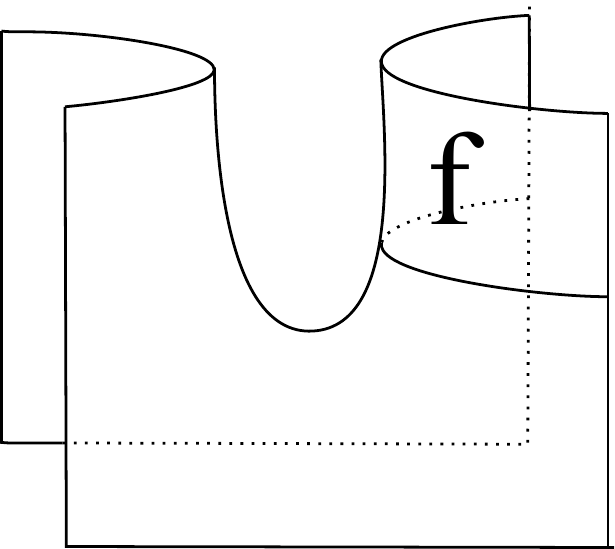} \ \ \ \ \ \ \ \ \ \ \ \ = \pic{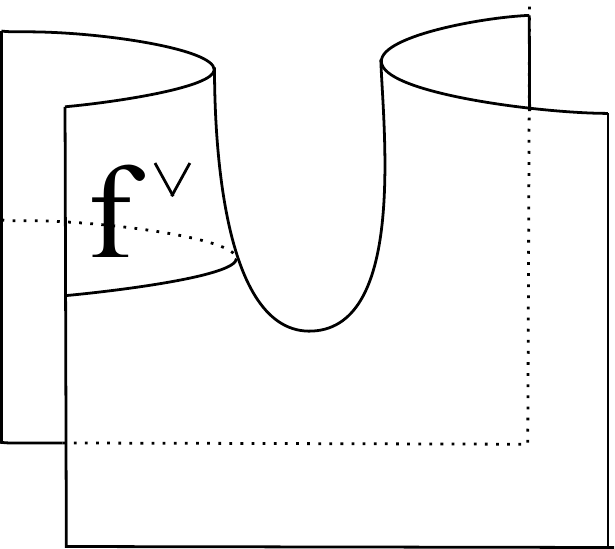} \ \ \ \ \ \ \ \ \ \ \ \ , \text{ and } \ \ \ \  \pic{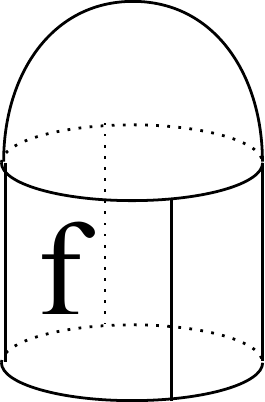} \ \ \ \ = \pic{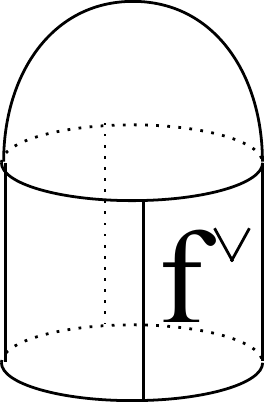} \ \ \ \ .
\]
We then define maps $\phi:\Hom(a,b)\rightarrow \Hom(\emptyset, b\otimes a^\vee)$ and $\psi:\Hom(\emptyset,b\otimes a^\vee)\rightarrow \Hom(a,b)$ by

\[
\phi(f) = \pic{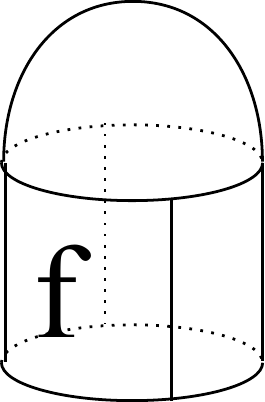} \ \ \ \ \ \ \ \ \hskip.4in \psi(\zeta) = \pic{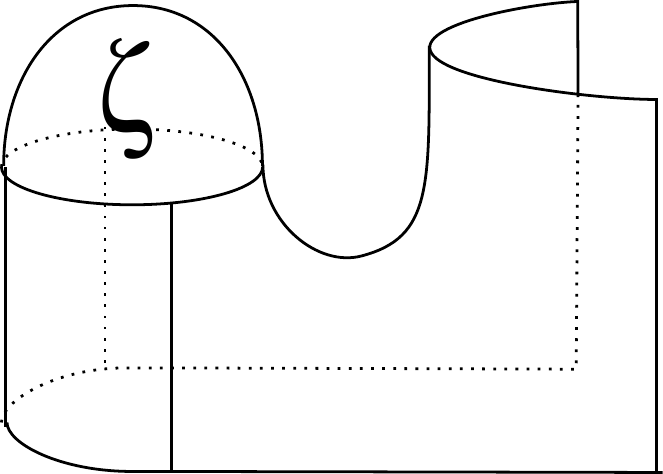}\ \ \ \ \ \ \ \ \ \ \ \ \ \ .
\]
That $\phi$ and $\psi$ are inverse isomorphism can be seen schematically:
\[
\pic{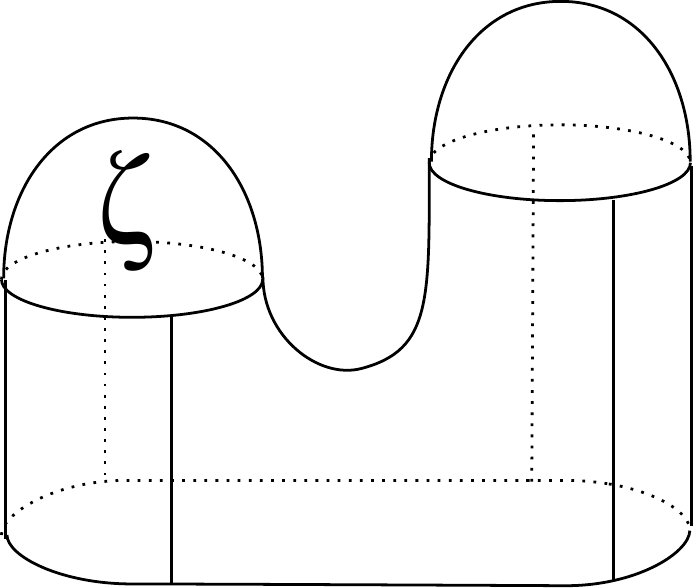}\ \ \ \ \ \ \ \ \ \ \ \ \ = \pic{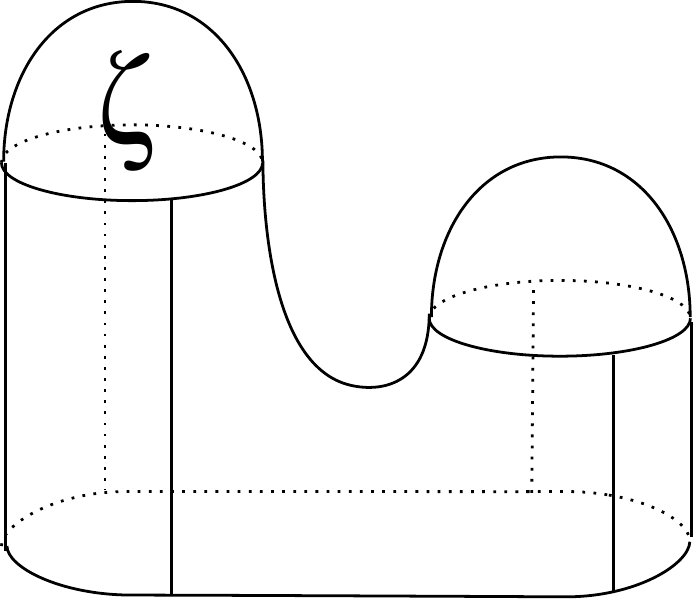}\ \ \ \ \ \ \ \ \ \ \ \ \ \  = \pic{zetaCap}\ \ \ , \ \ \ \ \ \ \pic{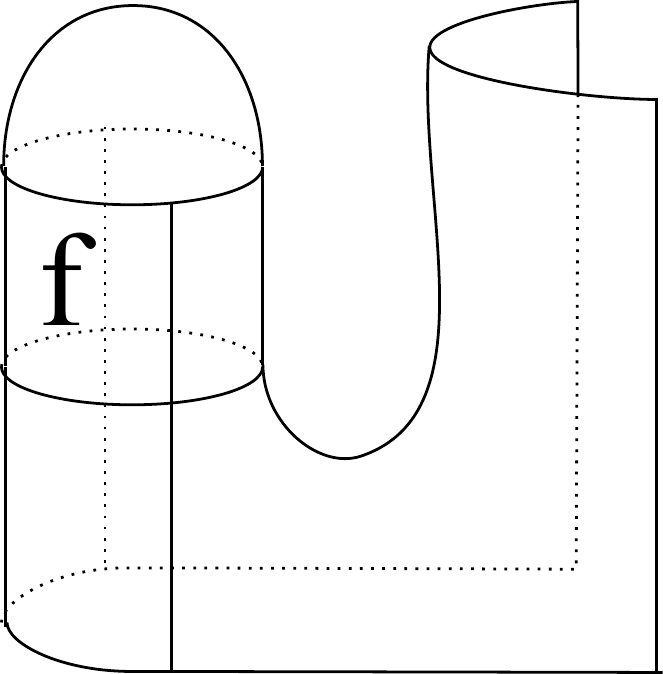}\ \ \ \ \ \ \ \ \ \ \ \ \ = \pic{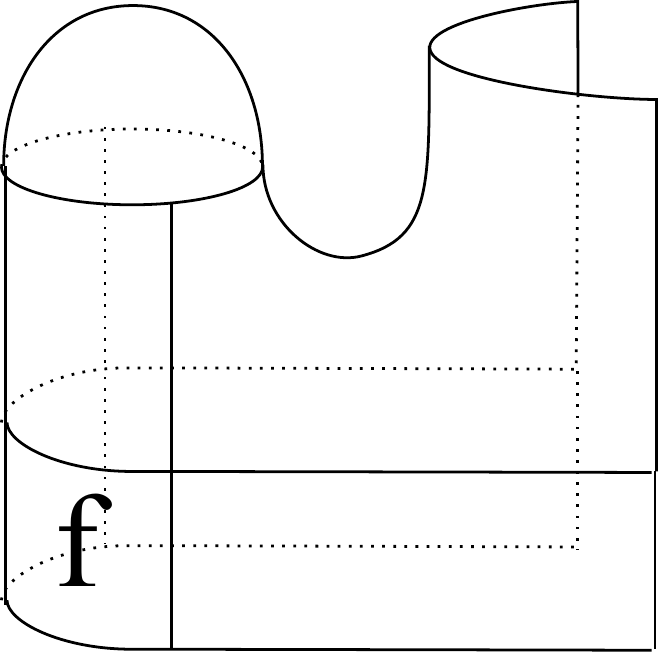}\ \ \ \ \ \ \ \ \ \ \ \ \   = \pic{verticalSheetWithF}\ \ \ .
\]
\end{remark}

\begin{proposition}\label{unitAndSaddleProp}
Fix indecomposable diagrams $a,b\in\BN{0}{2n}$ and $f:a\rightarrow b$.  We have
\begin{enumerate}
\item $(\Id_a \otimes s_a)\circ(\eta_a\otimes \Id_a) = \Id_a$ and $(s_a\otimes \Id_{a^\vee})\circ(\Id_{a^\vee}\otimes \eta_a) = \Id_{a^\vee}$
\item $s_b\circ (\Id_{b^\vee}\otimes f) = s_a\circ(f^\vee\otimes \Id_a)$
\item $(f\otimes \Id_{a^\vee})\circ \eta_a = (\Id_{a}\otimes f^\vee)\circ \eta_b$
\end{enumerate}
\end{proposition}
Before proving we remark that, morally speaking, the saddle maps $s_a$ make $a$ into a left $\ket{a\otimes a^\vee}$-module (and $a^\vee$ into a right $\ket{a\otimes a^\vee}$-module).  Part (1) of the proposition is the statement that $\eta_a\in\ket{a\otimes a^\vee}$ acts as the identity.

We may also think of the saddle maps as giving a family of bilinear forms $a^\vee\otimes a\rightarrow 1_n$.  Part (2) of the proposition could be interpreted as saying that $f$ and $f^\vee$ are adjoint with respect to this form.  Part (3) is a formal consequence of (1) and (2).
\begin{proof}
(1)  It is clear from the definitions that the 0-handles in $\eta_a$ cancel the 1-handles in $s_a$, so that  $(\Id_a \otimes s)\circ(\eta_a\otimes \Id_a)\cong \Id_a$ as abstract cobordisms.  With a little thought, it can be seen that they are isotopic in the cylinder, hence equal in $\End_{\BN{0}{2n}}(a)$.  The second statement is proven similarly.

(2)  By linearity, it suffices to assume that $f$ is a cobordism decorated with dots.  If (2) holds for $f$ and $g$, it clearly holds for $f\circ g$.  So by decomposing $f$ into elementary pieces, it suffices to assume that $f$ is a saddle or a dot.  If $f$ is a dot then (2) holds by sliding dots.  If $f$ is a saddle---i.e.\ a cobordism given by attaching a single 1-handle to $a\times I$---then (2) holds by isotopy invariance.

(3)  Tensor (2) on the left with $b$ and on the right with $a^\vee$ and precompose with $\eta_b\otimes \eta_a$ to obtain
\[
(\Id_b\otimes s_b\otimes \Id_{a^\vee})\circ (\Id_b\otimes \Id_{b^\vee}\otimes f\otimes \Id_{a^\vee}) \circ (\eta_b\otimes \eta_a)= (\Id_b\otimes s_a\otimes \Id_{a^\vee}) \circ(\Id_b\otimes f^\vee\otimes \Id_a\otimes \Id_{a^\vee})\circ (\eta_b\otimes \eta_a)
\]
On the left-hand side, the $\eta_b\otimes \Id_a\otimes \Id_{a^\vee}$ commutes past the middle term and cancels the first term (using (1)).  On the right-hand side, the $\Id_b\otimes \Id_{b^\vee} \otimes \eta_a$ commutes past the middle term and cancels the first term (again using (1)).  After this simplification we obtain $ (f\otimes \Id_{a^\vee})\circ \eta_a = (\Id_b\otimes f^\vee)\circ \eta_b$, which is (3).
\end{proof}
\begin{proof} (of proposition \ref{ordinaryDualityIso}).
Assume first that $a,b\in\BN{0}{2n}$ are diagrams with no circle components.  The case where $a,b\in\BN{0}{2n}$ are arbitrary follows at once from this case, since any object of $\BN{0}{2n}$ is a direct sum of such diagrams with shifts.  Define degree zero maps $\phi:\Hom_{\BN{0}{2n}}(a,b)\leftrightarrow q^n\ket{b\otimes a^\vee}:\psi$ by $\phi(f) = (f\otimes \Id_{a^\vee})\circ \eta_a$ and $\psi(\zeta)=(\Id_b\otimes s_a)\circ (\zeta\otimes  \Id_a)$.  Then $\phi$ and $\psi$ are inverses:
\begin{eqnarray*}
\phi(\psi(\zeta))
&=& (\psi(\zeta)\otimes \Id_{a^\vee})\circ \eta_a\\
&=&(\Id_b\otimes s_a\otimes \Id_{a^\vee})\circ (\zeta \otimes \Id_a\otimes \Id_{a^\vee})\circ \eta_a\\
&=&(\Id_b\otimes s_a\otimes \Id_{a^\vee})\circ (\Id_b\otimes \Id_{a^\vee}\otimes \eta_a)\circ \zeta\\
&=&(\Id_b\otimes \Id_{a^\vee})\circ \zeta\\
&=& \zeta
\end{eqnarray*}
and
\begin{eqnarray*}
\psi(\phi(f))
&=& (\Id_b\otimes s_a)\circ (\phi(f)\otimes  \Id_a)\\
&=&  (\Id_b\otimes s_a)\circ (f\otimes \Id_{a^\vee}\otimes \Id_a)\circ (\eta_a\otimes  \Id_a) \\
&=&  f\circ  (\Id_a \otimes s)\circ(\eta_a\otimes \Id_a)\\
&=&  f
\end{eqnarray*}

Now, suppose we have maps $a{\buildrel f\over \rightarrow }b{\buildrel g\over \rightarrow c}$.  Then $\phi(g\circ f) =  ((g\circ f)\otimes \Id_{a^\vee})\circ \eta_a = (g \otimes \Id_{a^\vee})\circ (f\otimes \Id_{a^\vee}) \circ \eta_a$.  On one hand, this gives $\phi(g\circ f) = (g\otimes \Id_{a^\vee}) \circ \phi(f)$.  On the other hand, using part (3) of proposition, this is also equal to 
\[
\phi(g\circ f) = (g\otimes \Id_{a^\vee})\circ  (\Id_b\otimes f^\vee)\circ \eta_b =   (\Id_c\otimes f^\vee)\circ (g\otimes \Id_{b^\vee})\circ \eta_b = (\Id_c\otimes f^\vee)\circ \phi(g)
\]
This implies naturality and completes the proof of the proposition.
\end{proof}

\subsection{Differential graded duality}

\begin{definition}[The tautological TQFT]\label{tautTQFT}
Abusing notation slightly, let $\ket{\ \ }:\Ch(\bn{0})\rightarrow \Ch(\Z-\mathbf{mod})$ denote the functor $\ket{C} = \Hom^\bullet_{\bn{0}}(\emptyset, C)$, as well as obvious extensions to complexes over $\bn{0}^\Pi$ and $\bn{0}^\oplus$.
\end{definition}

\begin{theorem}\label{dualityIsomorphism}
We have isomorphisms $\phi_{A,B}:\Hom^\bullet_{\BN{m}{n}}(A,B) \cong q^{(n+m)/2}\ket{B\cotimes A^\vee}$ which are natural in the sense that $\phi_{A',B'}\circ (L_g\circ R_f) = \phi_{A,B}\circ \ket{g\cotimes f^\vee}$.
\end{theorem}

\begin{proof}
First let us assume $m=0$.  Let $A,B\in\Ch(\BN{0}{2n})$ be arbitrary.  Note that
\begin{eqnarray*}
\Hom^k_{\bn{n}}(A,B)
&\cong & \prod_{j+i=k}\Hom_{\BN{0}{2n}}(A^{-i},B^j)\\
&\cong & \prod_{j+i=k}q^n\Hom_{\bn{0}}(\emptyset, B^j\otimes A^{-i})\\
&\cong & q^n\Hom_{\bn{0}}(\emptyset, \prod_{j+i=k} B^j\otimes A^{-i})\\
&\cong &  q^n\Hom_{\bn{0}}(\emptyset, (B \cotimes A^\vee)^k)
\end{eqnarray*}
This last group is precisely the $k$-th homogeneous piece of $\ket{B\cotimes A^\vee}$, so the theorem holds at least on the level of bigraded abelian groups.  The first isomorphism holds by definition of $\Hom^\bullet$, in the second we appealed to proposition \ref{ordinaryDualityIso}, in the third we used the universal property of direct products, and the last follows from the definitions of the differential graded versions of $(\ )^\vee$ and $\otimes$.

By proposition \ref{homAsFunctor} differential on $\Hom^\bullet_{\BN{0}{2n}}(A,B)$ is $\L{d_B} - \R{d_A}$.  Examining the definitions, we see that the differential on $\ket{B\cotimes A^\vee}$ is $\ket{d_B\cotimes \Id} - \ket{\Id\cotimes d_A^\vee}$.  So the naturality statement will imply that $\Hom^\bullet_{\BN{0}{2n}}(A,B)\cong q^n\ket{B\cotimes A^\vee}$ as a chain complexes.  It therefore remains only  to check $g\cotimes \Id$ corresponds to $R_g$ and $\Id\cotimes f$ corresponds to $L_f$ under the isomorphism of abelian groups $\Hom^\bullet(A,B) \cong q^n\ket{B\cotimes A^\vee}$.

Fix $g\in \Hom^k_{\BN{0}{2n}}(B,C)$, and let $g_j$ denote the restriction of $g$ to $B^j$.  The following square commutes, for all $i,j\in\Z$ by the naturality statement in proposition \ref{ordinaryDualityIso}:
\[
\begin{diagram}
\Hom_{\BN{0}{2n}}(A^i,B^j) & \rTo^{\cong} & q^n \ket{B^j\otimes (A^i)^\vee} \\
\dTo^{L_{g_j}} & & \dTo_{\ket{g_j\otimes \Id}}\\
\Hom_{\BN{0}{2n}}(A^i,C^{j+k}) & \rTo^{\cong} & q^n \ket{C^{j+k}\otimes (A^{i})^\vee}
\end{diagram}.
\]
The left downward arrow is the restriction of $L_g$ to $\Hom(A^i,B^j)$, and the right downward arrow is the restriction of $\ket{g\cotimes \Id}$ to $\ket{B^j\otimes (A^i)^\vee}$.  This shows that $L_g$ and $\ket{g\cotimes \Id}$ correspond to one-another.

Fix $f\in\Hom^k_{\BN{0}{2n}}(A,B)$, and let $f_i$ denote the restriction of $f$ to $A^i$.  Again, the following square commutes for all $i,j\in \Z$ by the naturality statement in proposition \ref{ordinaryDualityIso}:
\[
\begin{diagram}
\Hom_{\BN{0}{2n}}(B^i,C^j) & \rTo^{\cong} & q^n\ket{C^j\otimes (B^i)^\vee}\\
\dTo^{R_{f_{i-k}}}& & \dTo_{\ket{\Id\otimes (f_{i-k})^\vee}}\\
\Hom_{\BN{0}{2n}}(A^{i-k},C^j)  & \rTo^{\cong} & q^n\ket{C^j\otimes (A^{i-k})^\vee}
\end{diagram}.
\]
The leftmost arrow agrees with the restriction of $R_f$ to $\Hom_{\BN{0}{2n}}(B^i,C^j)$ up to a sign.  Examining the definitions, we see that the sign is $(-1)^{k(j-i)}$ since the degree of $f$ is $k$ and the degree of an element of $\Hom_{\BN{0}{2n}}(A^i,B^j)\subset \Hom^\bullet_{\BN{0}{2n}}(A,B)$ is $j-i$.  

The rightmost arrow agrees with the restriction of $\ket{\Id\cotimes f^\vee}$ to $C^j\otimes (B^i)^\vee$ up to a sign.  The sign is $(-1)^{jk}(-1)^{ik}$ (the first factor comes from the Koszul sign rule for $(\Id\cotimes f^\vee)$, and the second comes from the definition of $f^\vee$).  Since this sign is the same as the one in the previous paragraph, it follows that $R_f$ corresponds to $\ket{\Id\cotimes f^\vee}$.  This completes the proof in case $n=0$.

Finally, the result for general $n$ follows from the result for $n=0$: 
\[
\Hom^\bullet(A,B)\cong \Hom^\bullet\Big(\pic{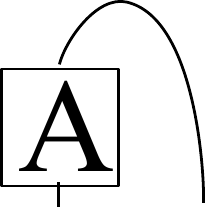}\ \  ,\pic{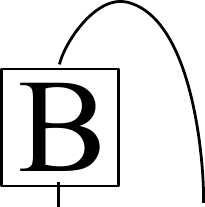}\ \ \Big) \cong q^{(n+m)/2}\Hom^\bullet\Big(\emptyset, \pic{B-Adual_trace}\ \ \  ^\Pi\Big),
\]
naturally.  This completes the proof.
\end{proof}

\subsection{Some restatements of the duality isomorphism}

The categories $\Ch(\bn{n})$ form a \emph{rigid monoidal category} \cite{BK01}:

\begin{corollary}\label{homRotation}
$\Hom^\bullet_{\BN{m}{n}}(A,B)$ is naturally isomorphic to any of the following chain complexes:
\begin{enumerate}
\item $q^{(m-n)/2}\Hom^\bullet_{\bn{m}^\Pi}(1_m, B\cotimes A^\vee)$
\item $q^{(n-m)/2}\Hom^\bullet_{\bn{n}^\Pi}(1_n, A^\vee\cotimes B)$
\item $q^{(m-n)/2}\Hom^\bullet_{\bn{m}^\oplus}(B\otimes A^\vee, 1_m)$
\item $q^{(n-m)/2}\Hom^\bullet_{\bn{n}^\oplus}(A^\vee\otimes B,1_n)$\qed
\end{enumerate}
\end{corollary}

In particular, setting $B=A$ in this corollary, we obtain maps in each of these four $\Hom^\bullet$ spaces corresponding to the identity $A\rightarrow A$.  As special cases we obtain maps $\eta_A$ and $s_A$ which generalize the maps $\eta_a$ and $s_a$ of section \ref{ordinaryDualitySection}.  Inspired by the original definition of the duality isomorphism in that section, we can phrase the duality isomorphism in general in terms of $\eta_A$ and $s_A$.  We choose a slightly different version, since this will be used directly later.

\begin{lemma}\label{epsilonLemma}
Let $A\in\Ch(\BN{m}{n})$ be arbitrary, let $\phi:\End(A)\rightarrow  q^{(m-n)/2}\Hom(\Tr(A\otimes A^\vee), \emptyset)$ be the isomorphism implied by corollary \ref{homRotation} (replacing $A$ by $\pic{A_turnedDown}\ \ $), and put $\e:=\phi(\Id_A)$.  Then
\[
\phi(f) = \e\circ \Tr(f\otimes \Id) = \e\circ \Tr(\Id\otimes f^\vee).
\]
\end{lemma}
\begin{proof}
This follows immediately from naturality of the isomorphism $\phi$.  That is to say, letting $\bra{\ \ }$ denote the functor $C\mapsto \Hom^\bullet_{\bn{0}}(C,\emptyset)$, $f\mapsto R_f$ (definition \ref{homAsFunctor}), we have $\phi\circ L_f = \bra{\Id\otimes f^\vee}\circ \phi$ and $\phi\circ R_f = \bra{f\otimes \Id}\circ \phi$, where $\bra{\ \  } = \Hom^\bullet(\ \  ,\emptyset)$.  Evaluating on $\Id_A$ gives $\phi(f) = \e\circ \Tr(\Id\otimes f^\vee)$ and $\phi(f) = \e\circ \Tr(f\otimes \Id)$, respectively.
\end{proof}

\section{Spin networks and morphisms}\label{calculusSection}

\subsection{Categorification of $p_n\in\TL_n$}
In \cite{CK12a} Cooper and Krushkal define a chain complex $P_n\in\Ch^{\leq 0}(\bn{n})$ which categorifies the idempotent $p_n\in \TL_n$.  We recall the relevant definitions and results in that paper, taking liberties to suit our needs here.

\begin{definition}
Put
\[
a_i := \overbrace{\bpic{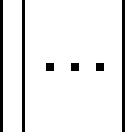}}^{i}\bpic{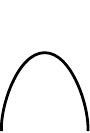} \overbrace{\bpic{vertStrands}}^{\mathclap{n-i-2}} 
\hskip.5in
b_j :=\overbrace{\bpic{vertStrands}}^{j}\bpic{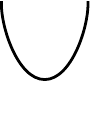}\overbrace{\bpic{vertStrands}}^{\mathclap{m-j-2}}.
\]
Say $Q\in \Ch(\BN{n}{m})$ \emph{kills turnbacks from above} if $a_i\otimes Q\simeq 0$ for $0\leq i\leq n-2$, similarly $Q$ \emph{kills turnbacks from below} if $Q\otimes b_j\simeq 0$ for $0\leq j
\leq m$.  Say $Q$ \emph{kills turnbacks} if it kills turnbacks from above and below.
\end{definition}

The following differs from Definition 3.1 in \cite{CK12a} in that a universal projector here is concentrated in non-positive homological degrees, rather than non-negative.  Also, we omit the condition on quantum grading here (which is only necessary to have a well-defined notion of Euler characteristic).
\begin{definition}
A \emph{universal projector} is a chain complex $P_n\in\Ch^{\leq 0}(\bn{n})$ such that
\begin{enumerate}
\item The degree zero chain group is $(P_n)^0 = 1_n$, the monoidal identity.  Moreover, $1_n$ does not appear as a summand of any other chain group.
\item $P_n$ kills turnbacks.
\end{enumerate}
\end{definition}

The following appears as Theorem 3.2, Corollary 3.4, and Corollary 3.5 in \cite{CK12a}.
\begin{theorem}[Cooper-Krushkal]\label{CKprojector}  Universal projectors exist and are unique up to homotopy equivalence.  Such a complex is also idempotent up to homotopy: $P_n \otimes P_n \simeq P_n$.
\end{theorem}

We will fix once and for all universal projectors $P_n$ for all $n=0,1,\ldots$.

\begin{example}\label{2projector}
For $P_2$ we may choose the chain complex 
\[
\begin{diagram}
P_2 &:= &\bigg( \cdots & \rTo^{\bdifferenceOfDots} & q^5 \bturnback & \rTo^{\bsumOfDots} & q^3 \bturnback & \rTo^{\bdifferenceOfDots} & q \bturnback & \rTo^{\bisaddle} & \underline\bstraightthrough \bigg)
\end{diagram}.
\]
It is a fun and useful exercise to show that $P_2$ is indeed killed by turnbacks.
\end{example}

The definition of spin networks \cite{KL94} extends immediately to the categorification.

\begin{notation}For an integer $n\geq 0$, denote $n$ parallel strands by $\pic{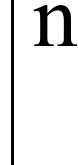}$.  Likewise, denote $P_n=:\pic{nProjector}$ and $P_n^\vee =:\pic{nProjectorDual}$.  If $a,b,c$ are nonnegative integers satisfying
\begin{itemize}
\item[(i)] $a+b+c\in 2\Z$ and
\item[(ii)] the sum of any two is larger than the third,
\end{itemize}
then let $\bpic{3network} \ \ \  $ denote $\bpic{3vertex}\ \ \ \ \ \ \in\Ch^{\leq 0}(\BN{a}{b+c})$.  The conditions on $a,b,c$ ensure there is a unique way to connect the projectors in the plane.
\end{notation}

\begin{definition}
A \emph{categorified spin network} is any chain complex over $\BN{n}{m}$ corresponding to a union of labelled graphs $\bpic{3network} \ \ \  $ along similarly colored boundary vertices.
\end{definition}

\subsection{Graphical calculus}
Theorem \ref{dualityIsomorphism} says that we can compute the space of morphisms between categorified spin networks in terms of certain planar compositions of $P_n = \pic{nProjector}$ and $P_n^\vee = \pic{nProjectorDual}$.  We give some rules which can be used to simplify many such compositions, as well as examples illustrating the danger of mistreating them.

\begin{proposition}\label{projSymmetries}
 $P_n$ has the symmetries of a rectangle.  I.e.~$s_x(P_n)\simeq s_y(P_n)\simeq P_n$, where $s_x$ and $s_y$ are the covariant functors induced by reflection of diagrams about the $x$-axis and $y$-axis respectively.
\end{proposition}
\begin{proof}
It is easy to see that $s_x(P_n)$ and $s_y(P_n)$ satisfy the axioms for universal projectors (kill turnbacks and the identity diagram appears exactly once and in bidegree $(0,0)$).  The proposition now follows from uniqueness of universal projectors (Theorem \ref{CKprojector}).
\end{proof}

\begin{corollary}\label{isotopyCorollary}
Let $M$ and $N$ be planar compositions of $P_n$'s for various $n$.  If the underlying diagrams (a union of arcs and rectangles in the disk) for $M$ and $N$ are isotopic rel boundary, then $M$ and $N$ are canonically homotopy equivalent.
\end{corollary}
\begin{proof}
From proposition \ref{projSymmetries}, we have $r(P_n) = s_x(s_y(P_n))\simeq P_n$.  Graphically, this is $\pic{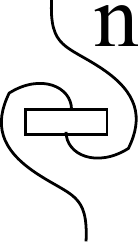}\ \simeq \pic{nProjector}$.  Now, proposition \ref{standardEquivProp} implies that there is in fact a canonical equivalence, up to homotopy.  This fact together with isotopy invariance of the underlying categories $\BN{m}{n}$, gives the result.
\end{proof}

\begin{definition}
Let $c\in \Cob^m_n\subset \BN{m}{n}$ be any object.  Define the \emph{through-degree} of $c$, denoted $\tau(c)$, to be the minimal $k$ such that $c = a\otimes b$, with $a\in\BN{m}{k}$, $b\in\BN{k}{n}$.   Define the through-degree of a chain complex $C\in\Ch(\BN{m}{n}^\Pi)$ or $\Ch(\BN{m}{n}^\oplus)$ to be $\tau(C) = \max\{\tau(c)\}$, where $c$ ranges over all direct summands (or factors) of all chain groups of $C$.
\end{definition}

It is clear that $\tau(A^\vee) = \tau(A)$ and $\tau(A\otimes B),\tau(A\cotimes B)\leq \min\{\tau(A),\tau(B)\}$.  All of the results in the remainder of this section are consequences of
\begin{enumerate}
\item The projectors $P_n$ kill turnbacks.
\item $P_n$ can be written as a mapping cone $P_n = \Cone(N\buildrel f\over \rightarrow 1_n)$ with $\tau(N)<n$.
\item The following lemma:
\end{enumerate}

\begin{lemma}[Turnback killing lemma]\label{turnbackLemma}
Suppose $Q\in\Ch(\BN{n}{m})$ kills turnbacks from above and $N\in \Ch(\BN{k}{n})$ has through degree $\tau(N)<n$.  If $Q$ is bounded below or $N$ is bounded above, then $N\otimes Q\simeq 0$.  The same is true is we allow $N\in\Ch(\BN{k}{n}^\oplus)$.
\end{lemma}
\begin{proof}
Refer to the definitions and results of section \ref{bicomplexesSection}.  Up to a shift, we may assume that $Q$ is non-negatively graded or $N$ is non-positively graded in homological degree.  The chain complex $N\otimes Q$ is the total complex (of type I, i.e.\ $\Tot^\oplus$) of the bicomplex $N^\bullet\otimes Q^\bullet := ((N^i\otimes Q^j)_{i,j\in\Z}, d_N\otimes \Id, \Id\otimes d_Q)$.  By assumption each $N^i$ is a direct sum $\bigoplus_x a_x$ where $a_x$ is a diagram with $\tau(a_x)<n$.  The columns of $N^\bullet \otimes Q^\bullet$ are the complexes $N^i\otimes Q\cong \bigoplus_x a_x \otimes Q$, which are contractible since $Q$ kills turnbacks.  If either $Q$ is non-negatively graded or $N$ is non-positively graded (in homological degree), then the bicomplex $N^\bullet\otimes Q^\bullet$ is concentrated in quadrants I, II, III, and proposition \ref{contractibleBicomplexes} implies that $N\otimes Q\simeq 0$.
\end{proof}

The boundedness conditions on $Q$ or $N$ in the above proposition cannot be relaxed, as the following example indicates.
\begin{example}\label{bi-infiniteExample}
Let $Q = P_2$, $N = q\inv \underline{\turnback}{\buildrel\phi_-\over \rightarrow} q^{-3} \turnback {\buildrel\phi_+\over\rightarrow} \dots$, where the differential alternates between a difference and a sum of dots, and the underlined term lies in homological degree zero.  In other words $N$ is the tail of $\pic{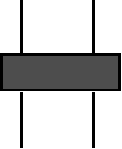}$, so that $\pic{2projectorDual} = t\Cone(\straightthrough\rightarrow N)$.  Then $Q$ kills turnbacks and $\tau(N)=0$.  However, $Q\otimes N$ is not contractible.
Indeed, $\pic{2projectorDual} = t\Cone(\straightthrough\rightarrow N)$ implies that $N\simeq \Cone(\pic{2projectorDual}\rightarrow \straightthrough)$, and so
\[
N\otimes Q \simeq \Cone(\pic{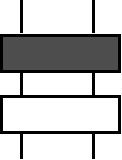}^\oplus\rightarrow \pic{2projector}) \simeq \Cone(\pic{2projectorDual}\buildrel\phi\over\rightarrow \pic{2projector}).
\]
Here we are already using the result of proposition \ref{iotaProp}.  This chain complex cannot possibly be contractible, since the map $\phi$ is supported in a single homological degree.  It turns out that this map is the identity on the degree zero chain groups.  Contracting this isomorphism, we obtain a 2-periodic chain complex:
\[
Q\otimes N \simeq
\Big( \begin{diagram}
\cdots & \rTo^{\bdifferenceOfDots} & q \bturnback & \rTo^{\bsumOfDots} & q\inv \underline{\bturnback} & \rTo^{\bdifferenceOfDots} & \cdots
\end{diagram} \Big)
\]
This chain complex is very far from being contractible.  On the contrary, it ``measures the difference'' between $\pic{2projector}$ and $\pic{2projectorDual}$, and is interesting in its own right.
\end{example}

\begin{proposition}\label{iotaProp}
If $Q\in\Ch(\BN{n}{m}^\oplus)$ kills turnbacks from above, then $P_n\otimes Q\simeq Q$.  In fact,  $\iota\otimes \Id_Q: Q = 1_n\otimes Q\rightarrow P_n\otimes Q$ is a homotopy equivalence, and there is an inverse equivalence $\psi:P_n\otimes Q\rightarrow Q$ such that $\psi\circ (\iota\otimes \Id_Q)=\Id_Q$.  We have a similar fact if $Q$ kills turnbacks from below.
\end{proposition}
\begin{proof}
Suppose $Q\in \Ch(\BN{n}{m})$ kills turnbacks from above, and write $P_n = (N\rightarrow 1_n)$, where $N\in\Ch^{\leq 0}(\bn{n})$ has through degree $\tau(N)< n$.  Then $N\otimes Q\simeq 0$ by Lemma \ref{turnbackLemma}.  Proposition \ref{Gauss} now implies $P_n\otimes Q \cong (N\otimes Q\rightarrow Q) \simeq Q$.  In fact, the proof of proposition \ref{Gauss} actually implies that $\iota\otimes \Id_Q :Q\rightarrow  P_n\otimes Q$ is the inclusion of a strong deformation retract.  In particular there is an inverse $\phi$ such that $\phi\circ(\iota\otimes \Id_Q)=\Id_Q$.  This proves the proposition.
\end{proof}

The next three results constitute the most important relations in our graphical calculus.

\begin{proposition}[Absorption rule] \label{absorptionRule}
Fix non-negative integers $x,y,z$, and put $a:=x+y+z$.  Then we have
\[
\pic{littleWhite-bigWhite}\ \ \ \simeq \bpic{aProjector} \simeq \pic{littleBlack-bigWhite}\ \ \ \ ^\Pi \hskip.5in \text{and}\hskip.5in \pic{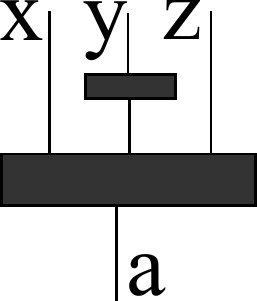}\ \ \ \simeq \bpic{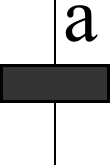} \simeq \pic{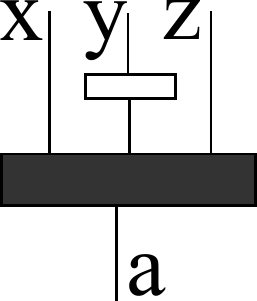}\ \ \ \ ^\oplus,
\]
and vertical reflections of these.
\end{proposition}
\begin{proof}
Regard $P_a$ as an object $Q\in \Ch(\BN{y}{x+a+z})$ by bending the $x$ top left-most strands to the left and down and the $z$ top right-most strands to the right and down.  Then $Q$ kills turnbacks from above.  By proposition \ref{iotaProp} we have $P_y\otimes Q\simeq Q$, which is the first equivalence in the statement above.  The remaining equivalences are proven similarly.
\end{proof}

\begin{proposition}[Commuting rule]\label{commutingRule}
Let $A\in \Ch(\bn{n}^\oplus)$ (respectively $A\in \Ch(\bn{n}^\Pi)$) be arbitrary.  We have
\[
\pic{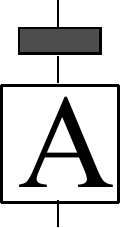}^\oplus\simeq \pic{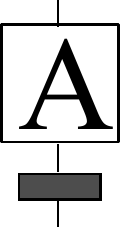}^\oplus, \hskip.5in \text{respectively}\hskip.5in  \pic{nProjector-A} ^\Pi \simeq \pic{A-nProjector}^\Pi .
\]
\end{proposition}
\begin{proof}
Fix $A\in\Ch(\bn{n})$, and let $a\in\BN{n-2}{n}$ be a turnback diagram.  Then $a\otimes A$ has through degree $<n$.  We have $a\otimes (A\otimes P_n^\vee)\cong  (a\otimes A)\otimes P_n^\vee\simeq 0$ by proposition \ref{iotaProp}.  So $A\otimes P_n^\vee$ kills turnbacks from above.  $A\otimes P_n^\vee$ clearly kills turnbacks from below (since $P_n^\vee$ does) so $A\otimes P_n^\vee$ kills turnbacks.  Similarly, $P_n^\vee \otimes A$ kills turnbacks.  Two applications of proposition \ref{iotaProp} give
\[
A\otimes P_n^\vee \simeq P_n^\vee\otimes (A\otimes P_n^\vee)\cong (P_n^\vee\otimes A)\otimes P_n^\vee \simeq P_n^\vee\otimes A.
\]
This proves the first statement.  The second follows by an application of $(\ )^\vee$.
\end{proof}

\begin{proposition}[Semi-orthogonality rule]\label{semiRule}
Suppose $i<j$ and let $A\in\Ch(\BN{i}{j}^\oplus)$ (respectively $A\in\Ch(\BN{j}{i}^\Pi)$) be arbitrary.  We have
\[
\pic{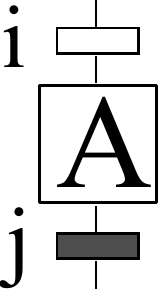}^\oplus\simeq 0 \hskip.5in \text{respectively }\hskip.5in \pic{jProjector-A-iProjectorDual}^\Pi\simeq 0 .
\]
\end{proposition}
\begin{proof}
Let $i,j,A$ be as in the hypotheses.  Note that $A\otimes P_i$ has through degree $\tau(A\otimes P_i)\leq i<j$.  Since $P_j^\vee$ is bounded below, lemma \ref{turnbackLemma} applies, and we have $P_j^\vee \otimes A \otimes P_i\simeq 0$.  This is the first statement.  The second statement follows from an application of $(\ \ )^\vee$.
\end{proof}

One must be careful not to mistreat the diagrammatic simplifications of this section.  The following example is useful to keep in mind:

\begin{example}
For this example specialize to $\a=0$.  That is to say, two dots on the same component of a surface gives the zero map.  Let $A\in \Ch(\bn{2})$ be the chain complex $(\dots\rightarrow t\inv q^3\turnback \rightarrow q\turnback \rightarrow tq\inv \turnback\rightarrow \dots)$ in which each map is a dot on the bottom strand.  It is not too hard to show that this chain complex kills turnbacks from below but not from above.  Therefore
\[
A\otimes P_2\simeq A \not \simeq P_2\otimes A,
\]
since the latter chain complex kills turnbacks from above as well as below, and the former does not.  This gives a counter-example to the statements ``$P_n\otimes A\simeq A\otimes P_n$ for $A\in\Ch(\bn{n})$'' and ``$P_n\otimes A\simeq 0$ for $\tau(A)<n$.''  In particular, the hypotheses of the commuting and semi-orthogonality rules (propositions \ref{commutingRule} and \ref{semiRule}, respectively) are necessary.
\end{example}

We conclude this section with an observation:

\begin{corollary}
Suppose $Q\in \Ch(\bn{n})$ is semi-infinite in homological degree.  If $Q$ kills turnbacks from above then $Q$ kills turnbacks from below, and vice versa.
\end{corollary}
\begin{proof}
Up to a shift, we may assume $Q\in \Ch^{\leq 0}(\bn{n})$ or $Q\in \Ch^{\geq 0}(\bn{n})$.  So, suppose $Q\in \Ch^{\leq 0}(\bn{n})$ kills turnbacks from above.  Then
\begin{equation}\label{symmetryEq}
Q\simeq P_n\otimes Q \cong P_n\cotimes Q \simeq Q\cotimes P_n.
\end{equation}
The first equivalence follows from proposition \ref{iotaProp}, the middle isomorphism follows since $\otimes $ and $\cotimes $ agree on the category of non-positively graded chain complexes, and the third equivalence follows from proposition \ref{commutingRule}.  The complex on the RHS of equation \ref{symmetryEq} obviously kills turnbacks from below as well as from above.  Similar arguments take care of the remaining cases.
\end{proof}

\begin{remark}
This result allows us to simplify the expression for the universal projector constructed in \cite{CK12a}.  The result of the Cooper-Krushkal \cite{CK12a} construction (up to the penultimate step) is a chain complex of the form
\[
P_n' =
\Big(\dots \rightarrow \pic{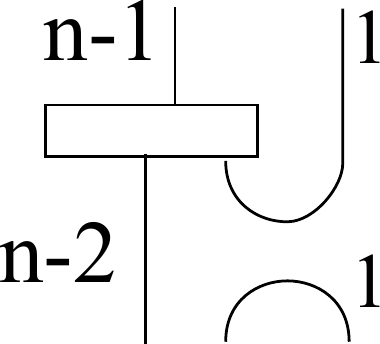}\ \ \ \ \ \rightarrow \pic{FKpic_1}\ \ \ \ \ \rightarrow \dots \rightarrow \pic{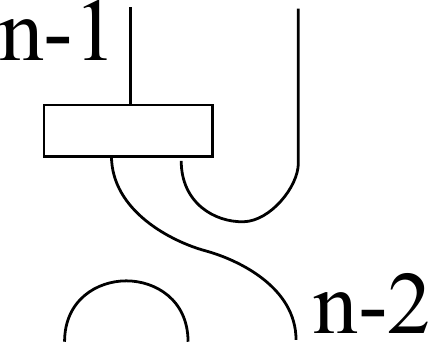}\ \ \ \ \ \rightarrow  \pic{FKpic_n-1}\ \ \ \ \ \rightarrow \dots \rightarrow  \pic{FKpic_1}\ \ \ \ \ \rightarrow \pic{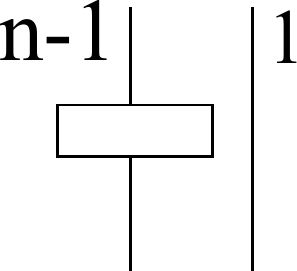}\ \ \ \ \ \Big).
\]
The differential should be regarded as a sum of arrows pointing strictly to the right; we did not draw the arrows of length $>1$, nor the degree shifts.  Cooper and Krushkal prove that this chain complex kills turnbacks from below, hence $P_n:=s_x(P_n')\otimes P_n'$ is a universal projector, where $s_x(\ )$ is the vertical reflection.  Once we know that a universal projector exists, proposition \ref{commutingRule} implies that we already had one at the previous step: $P_n'$ is bounded above in homological degree and kills turnbacks from below, so it kills turnbacks from above.  This fact would be hard to prove without a priori knowledge that $P_n$ exists.
\end{remark}

\subsection{Some computations}\label{computatonsSection}
In this section we will implicitly be working with categories $\BN{m}{n}^\Pi$, and so we will omit the symbol $\Pi$ from our planar compositions.  Our work up to this point says that we can compute the chain complex $\Hom^\bullet_{\BN{m}{n}}(M,N)$ of morphisms between planar compositions of $P_n$ in the following way:

\begin{enumerate}
\item Reflect $M$ and replace all the white boxes with black boxes to obtain $M^\vee$
\item Glue up all of the loose ends of $N$ with the corresponding loose ends of $M^\vee$
\item Simplify using the following relations:
\begin{enumerate}
\item Diagrams which are isotopic rel boundary give canonically homotopy equivalent chain complexes (corollary \ref{isotopyCorollary}).
\item Absorption rules: $\pic{littleWhite-bigWhite}\ \ \ \simeq \bpic{aProjector} \simeq \pic{littleBlack-bigWhite}\ \ \ \ $ (proposition \ref{absorptionRule}).
\item Commuting rule: $\pic{nProjector-A}\simeq\pic{A-nProjector}$ for every chain complex $A$ over $\bn{n}$ (proposition \ref{commutingRule}).
\item Semi-orthogonality rule: if $i<j$ then $\pic{jProjector-A-iProjectorDual}\simeq 0 $ for every chain complex $A$ over $\BN{i}{j}$ (proposition \ref{semiRule}).
\end{enumerate}
\item Take $\Hom^\bullet(\emptyset, -)$ of the result.
\end{enumerate}
The planar composition of complexes over $\BN{m}{n}^\Pi$ is spherical, in the sense that it doesn't matter up natural how we connect up the loose ends in (2) above, as long as the strands are paired correctly.

\begin{proposition} \label{endRings}
We have
\begin{enumerate}
\item $\End^\bullet(\pic{aProjector} ) \simeq q^a \ket{\pic{aProjector_trace}\ \  }$.
\item $\End^\bullet(\pic{3network}\ ) \simeq q^{(a+b+c)/2}\ket{\ \pic{thetaNet}\ \ }$.
\item $\End^\bullet(\ \pic{iNet_i}\ \ )\simeq q^{(a+b+c+d)/2} \ket{ \ \pic{pillNet}\ \ \ \ }$.
\end{enumerate}
\end{proposition}
\begin{proof}
Let us prove (3) only.  The other parts are special cases of this one.  Observe that in the category $\bn{0}^\Pi$ we have
\[
\pic{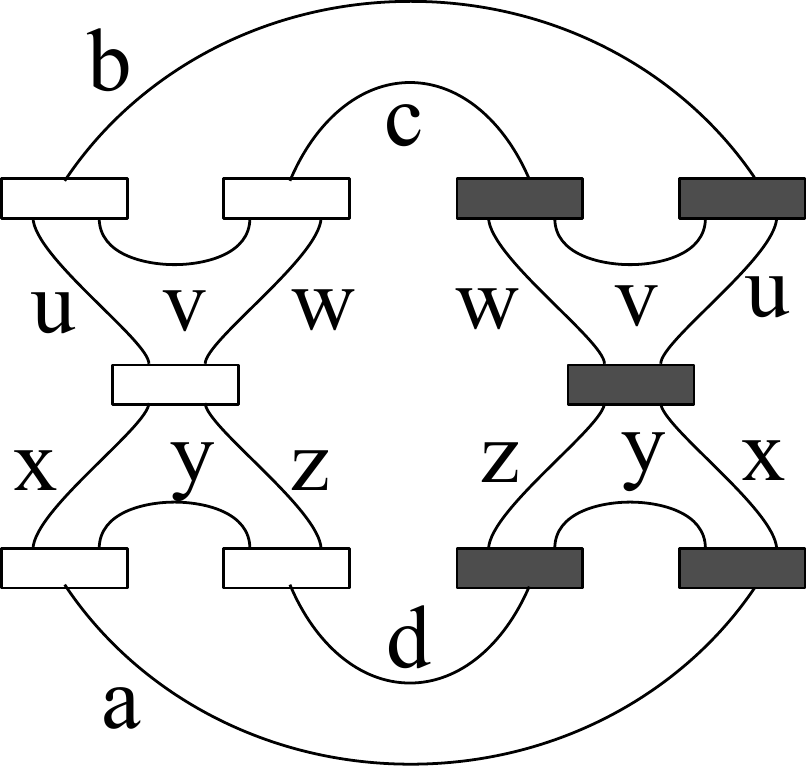}\hskip .8 in \simeq \  \pic{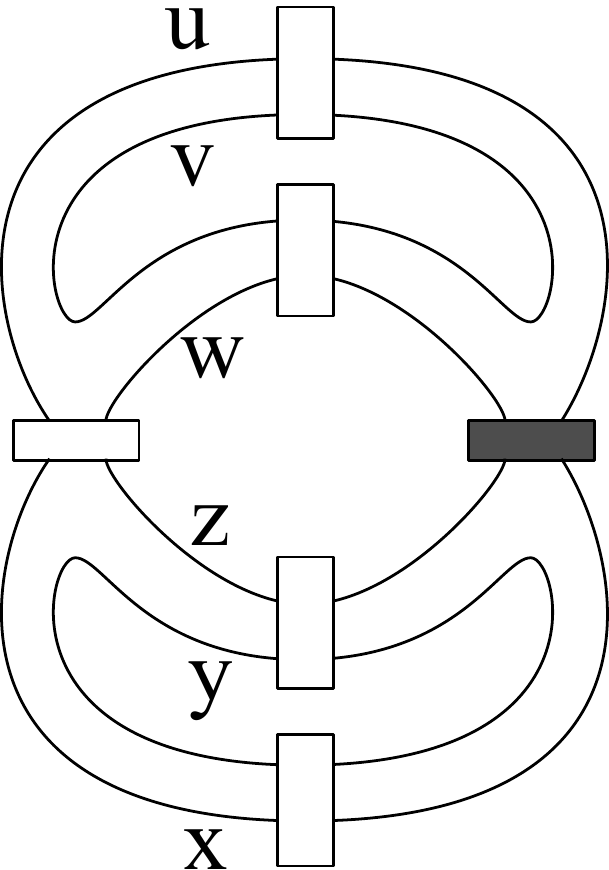}\hskip .6in \simeq \ \pic{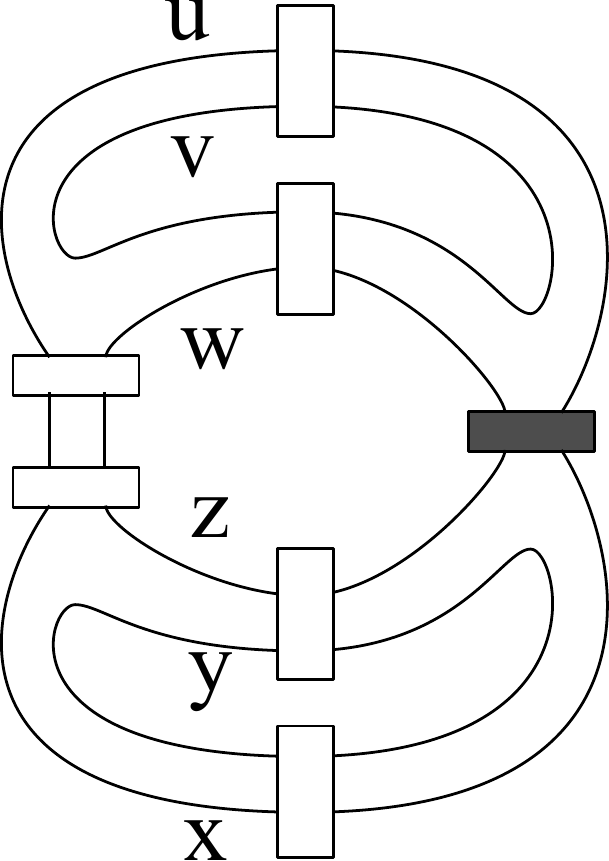}\hskip .6in \simeq\  \pic{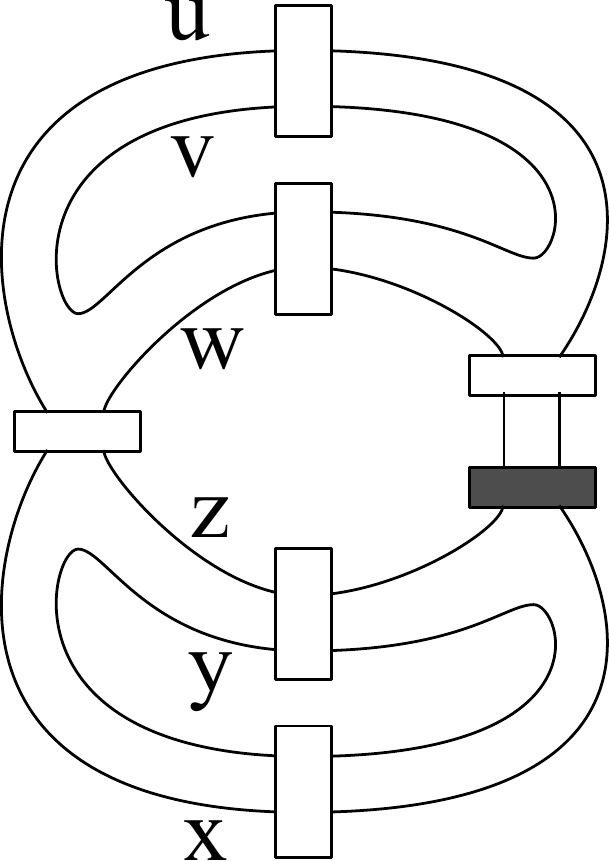}\hskip .6in \simeq\  \pic{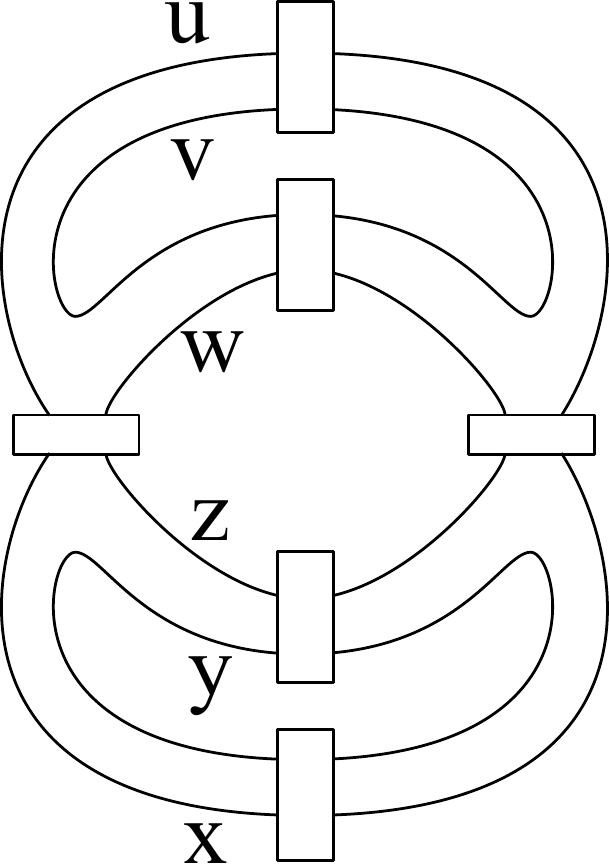}\hskip .6in,
\]
In the first equivalence we repeatedly used that $\pic{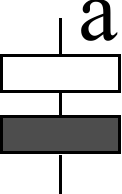}^\Pi \simeq \pic{aProjector}$ (proposition \ref{absorptionRule}), and in the third we the commuting rule  (proposition \ref{commutingRule}).  An application of $\ket{\ \ }$, together with the duality isomorphism, now gives the result.
\end{proof}
Note that the duality functor automatically gives an isomorphism $\End^\bullet_{\bn{n}}(A)\cong \End_{\bn{n}}^\bullet(A^\vee)$ for any $A\in\Ch(\bn{n})$.  For example $\End^\bullet(\pic{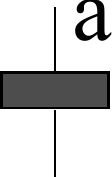})\simeq q^a\ket{\pic{aProjector_trace}\ \ }$.  This illustrates a preference for chain complexes which are bounded above, and is our reason for choosing our conventions the way we have: we want the colored unknots give the endomorphism rings of the colored arcs.

\begin{proposition}\label{iNetProp}
If $j<i$ then $\Hom^\bullet\Big(\ \pic{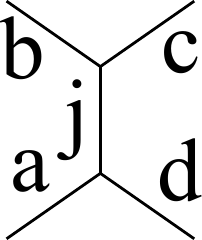}\ \ ,\ \pic{iNet_i}\ \ \Big)\simeq 0$.
\end{proposition}
\begin{proof}
Proposition \ref{semiRule} says that $ \ \pic{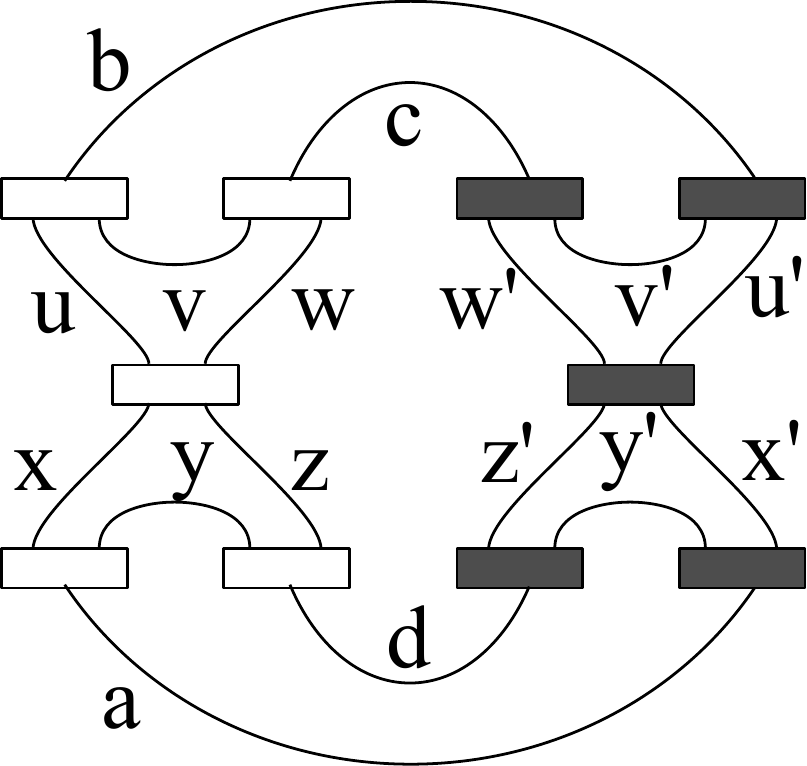}\ \ \ \ \ \ \ \ \  \ \ \ \ \   \simeq 0$ if $j:=x'+z'<x+z=:i$.  Theorem \ref{dualityIsomorphism} now gives the result.
\end{proof}

\begin{remark}
The complexes here are not bounded, so there is no reason to expect that $\Hom^\bullet(A,B)\simeq \Hom^\bullet(B,A)$ in general.  Nonetheless it can be shown that $\Hom^\bullet\Big(\ \pic{iNet_j}\ \ ,\ \pic{iNet_i}\ \ \Big)\simeq 0$ more generally if $j\neq i$.
\end{remark}
\section{The sheet algebra and colored unknots}\label{sheetAlgSection}
As an application of the $\Hom^\bullet$ space calculations thus far, we are able to study the differential graded algebra $\End(\pic{nProjector})\simeq q^n\ket{\pic{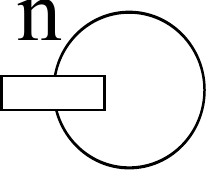}\ }$, which we refer to as the \emph{sheet algebra}.  Our main results in this section are that this algebra is (graded) commututive up to homotopy, and that the actions of the unknots $q^n\ket{\pic{nProjector_trace}\ }$ on $\pic{nProjector}$, $\pic{nProjectorDual}$, coincide with the actions induced by the saddle maps.

In this section, we fix $n\in\N$ and put $P=P_n$.  For simplicity of notation, we write $\End = \End^\bullet_{\bn{n}}$ and $\Hom = \Hom^\bullet_{\bn{n}}$.  Recall that our simplifications up to this point imply that $\End(P)\simeq q^n \ket{\Tr(P)}$, where $\Tr:\bn{n}\rightarrow \bn{0}$ is the functor obtained by connecting the top and bottom points of all the diagrams (the Markov trace) and $\ket{\ }:\Ch(\bn{0})\rightarrow \Z-\mathbf{mod}$ is the functor $A^\bullet \mapsto \Hom^\bullet_{\bn{0}}(\emptyset,A^\bullet)$.

\subsection{Self-equivalences of $P$}
We make some simple observations about $\End(P)$ and $\ket{\Tr(P)}$.
\begin{lemma}\label{degreeZeroLemma}
$H^{(0,0)}(\End(P))\cong \Z$.
\end{lemma}
\begin{proof}
This can easily be seen from the Cooper-Krushkal construction and the computation $\End(P)\simeq q^n \ket{\Tr(P)}$.
\end{proof}
It follows from this proposition that only self equivalences of $P$ up to homotopy are $\pm\Id_P$.  Indeed the same computation gives a slight improvement on a theorem of Cooper-Krushkal.
\begin{theorem}\label{canonicalEquivalencesThm}
Any two universal projectors are homotopy equivalent via a unique equivalence (up to homotopy and a scalar).\qed
\end{theorem}
We can fix the scalar once and for all using the following proposition.  Let $1_n\in\bn{n}$ be the monoidal identity (i.e.\ $n$ parallel strands).  If $P$ is a univeral projector, let $\iota_P:1_n\rightarrow P$ be the inclusion of the degree zero chain group (a chain map since $P$ is non-positively graded).
\begin{definition}\label{standardEquivDef}
Let $P$ and $Q$ are universal projectors.  Call an equivalence $\psi:P\simeq Q$ \emph{standard} if $\psi\circ\iota_P = \iota_Q$.
\end{definition}
The following proposition says that standard equivalences exist and are unique up to homotopy.

\begin{proposition}\label{standardEquivProp}
Let $P$ and $Q$ be universal projectors.  There is a map $\psi:P\rightarrow Q$ uniquely defined, up to homotopy, by $\psi\circ \iota_P = \iota_Q$.  This map is a homotopy equivalence.
\end{proposition}
\begin{proof}
For $\sigma\in \Hom^\bullet(1_n,Q)$, define $\hat \sigma \in \Hom^\bullet(P,Q)$ to be the composition $P{\buildrel  \Id\otimes \sigma \over \longrightarrow} P\otimes Q \buildrel\phi\over \rightarrow Q$, where $\phi:P\otimes Q\simeq Q$ is a homotopy equivalence such that $\phi\circ(\iota_P\otimes \Id_Q)=\Id_Q$ (implied by proposition \ref{iotaProp}).  Then
\[
\hat\sigma \circ \iota_P = \phi\circ (\iota_P\otimes\sigma) = \phi\circ(\iota_P\otimes \Id_Q) \circ \sigma = \sigma.
\]
This implies the chain map $R_{\iota_P}:\Hom^\bullet(P,Q)\rightarrow \Hom^\bullet(1_n,Q)$ is surjective.  Writing $P$ as a mapping cone $P = (N\rightarrow 1_n)$ with $\tau(N)<n$, we therefore have a short exact sequence of chain complexes
\[
\Hom^\bullet(N,Q)\buildrel{R_\pi}\over \longrightarrow \Hom^\bullet (P,Q)\buildrel R_{\iota_P}\over\longrightarrow \Hom^\bullet (1_n,Q),
\]
where $\pi: P\rightarrow N$ is the projection and $R_\pi$, $R_{\iota_P}$ are as in definition \ref{homAsFunctor}.
But $\Hom^\bullet(N,Q)\simeq q^n \ket{Q\cotimes N^\vee}\simeq 0$ by Lemma \ref{turnbackLemma}.  Examining the long exact sequence in homology gives that $R_{\iota_P}$ is an isomorphism in homology (even more is true: $f\mapsto f\circ \iota_P$ and $\sigma\mapsto \hat\sigma$ are homotopy inverses, but we don't need this here).  In particular, $\psi:=\hat\iota_Q$ is the unique map $P\rightarrow Q$ up to homotopy such that $\psi\circ\iota_P = \iota_Q$.

To see that this $\psi$ is a homotopy equivalence, let $\psi':Q\rightarrow P$ be a chain map such that $\psi\circ \iota_Q = \iota_P$, which exists by the above.  Then $(\psi'\circ\psi) \circ \iota_P = \psi'\circ \iota_Q = \iota_P$, which implies $\psi'\circ \psi\simeq \Id_P$ by the uniqueness statement above.  Similarly, $\psi\circ \psi'\simeq\Id_Q$, so $\psi$ and $\psi'$ are homotopy inverses.  This completes the proof.
\end{proof}

\subsection{Modules over the sheet algebra}
\begin{definition}
The \emph{sheet algebra} is the dg algebra $\End(P) = \End(\pic{nProjector})$.  By a \emph{sheet module} we will mean a homotopy $\End(P)$-module, i.e.~a chain complex $M$ and a degree zero chain map $\pi:\End(P)\rightarrow \End(M)$ such that $\pi(f\circ g) \simeq \pi(f)\circ \pi(g)$ as maps $\End(P)^{\otimes 2}\rightarrow \End(M)$.  We require that $\pi(\Id_P)\simeq \Id_M$.  We call the map $\pi$ the \emph{action} or the \emph{representation}.
\end{definition}
Any chain complex $Q\in\BN{n}{m}$ which kills turnbacks from above naturally has the structure of a sheet module.  Indeed, let $\iota:1_n\rightarrow P$ be the inclusion of the degree zero chain group.  Proposition \ref{iotaProp} says that $\iota\otimes \Id_Q:Q \rightarrow P\otimes Q$ is a homotopy equivalence, and there is an inverse $\phi:P\otimes Q\rightarrow Q$ such that $\phi\circ(\iota\otimes \Id_Q) = \Id_Q$.
\begin{definition}\label{sheetModuleConstruction}
Retain notation as above.  Let $\pi_Q:\End(P)\rightarrow \End(Q)$ be the map $\pi_Q(f) = \phi\circ (f\otimes\Id_Q)\circ (\iota\otimes \Id_Q)$.
\end{definition}
It is easy to check that $\pi_Q$ makes $Q$ into a sheet module.
\begin{example}
Fix $k\geq 0$, and regard $P_{n+k}$ as a chain complex over $\BN{n}{n+2k}$ (bending the $k$ rightmost top strands down).  Then these chain complexes kills turnbacks from above, and so we have actions of $\End(P_n)$ on $P_{n+k}$ and $P_{n+k}^\vee$.
\end{example}

\begin{lemma}\label{sheetModuleLemma}
Let $Q\in \BN{n}{m}$ kill turnbacks from above, and let $P = P_n$.
\begin{enumerate}
\item The map $\End(Q)\rightarrow   \End(P\otimes Q)$ sending $f\mapsto \Id_P\otimes f$ is a homotopy equivalence.
\item The two maps $\End(P)\rightarrow \End(P\otimes Q)$ sending $f\mapsto f\otimes \Id_Q$ and $f\mapsto \Id_P\otimes \pi_Q(f)$ are homotopic.
\item $\pi_P\simeq \Id_{\End(P)}$, i.e.~the two given actions of $\End(P)$ on $P$ are homotopic.
\item $\pi_{P^\vee}\simeq (f\mapsto f^\vee)$, i.e.~the two given actions of $\End(P)$ on $P^\vee$ are homotopic.
\end{enumerate}
\end{lemma}

\begin{proof}
(1).  Let $\iota = \iota_P:1_n\rightarrow P$ the the inclusion of the degree zero chain group, as usual.  By proposition \ref{iotaProp} we have standard equivalences $\iota\otimes \Id_Q : Q\rightarrow  P\otimes Q$ and $\psi:P\otimes Q\rightarrow Q$ such that $\psi \circ (\iota \otimes \Id_Q) = \Id_Q$.  The map $\Phi(F):= \psi\circ F \circ (\iota\otimes \Id_Q)$ is a homotopy equivalence $\Phi:\End(P\otimes Q)\simeq \End(Q)$ since it is pre- and post- composition with homotopy equivalences.  The map $\Phi'(f):=\Id_P\otimes f$ is a right inverse for $\Phi$ since $\Phi(\Phi'(f)) = \psi\circ (\Id_P\otimes f)\circ (\iota\otimes \Id_Q) = \psi\circ (\iota\otimes \Id_Q)\circ f = f$.  Since $\Phi$ is a homotopy equivalence, it follows that $\Phi'$ a homotopy inverse for $\Phi$.  This proves (1).

(2) Let $\Psi_1,\Psi_2: \End(P)\rightarrow \End(P\otimes Q)$ be the maps $\Psi_1(f) = f\otimes \Id_Q$, $\Psi_2(f) = \Id_P\otimes \pi_Q(f)$.  Post-composing with the homotopy equivalence $\Phi:\End(P\otimes Q)\simeq \End(Q)$ above gives $\Phi\circ \Psi_1(f) = \pi_Q(f)$ (by definition of $\pi_Q$) and $\Phi\circ \Psi_2(f) = \psi\circ( \iota\otimes \pi_Q(f)) = \pi_Q(f)$.  Since $\Phi$ is a homotopy equivalence, this shows $\Psi_1\simeq \Psi_2$ which is (2).

(3).  We have $\iota\otimes \Id_P\simeq \Id_P\otimes \iota$, since both are standard equivalences $P\simeq P\otimes P$ (definition \ref{standardEquivDef}).  Therefore $\pi_P(f) = \phi \circ (f\otimes \Id_P)\circ (\iota\otimes \Id_P)$ is homotopic to the map $\pi'(f) = \phi \circ (f\otimes \Id_P)\circ (\Id_P\otimes \iota) = \phi\circ (\Id_P\otimes \iota)\circ f$ which, in turn is homotopic to the identity since $\phi\circ (\Id_P\otimes \iota)\simeq \Id_P$.  This proves (3).

(4).  Let $\bra{\ \ }$ denote the functor $C\mapsto \Hom^\bullet(C,\emptyset)$, and let $\phi:\End(P^\vee)\cong q^n \bra{\Tr(P\otimes P^\vee)}$, $\e = \phi(\Id_{P^\vee})$, be as in lemma \ref{epsilonLemma} (with $A = P^\vee$).  Consider the following diagram:
\[
\begin{diagram}
\End(\pic{nProjector_white-black}^\oplus) & \lTo^{\simeq } & \End(\pic{nProjectorDual})\\
\dTo^{\Tr} & & \dTo_{\phi}\\
\End(\pic{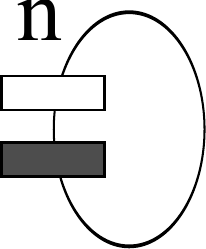}\ ^\oplus)  & \rTo^{L_\e} & q^n\bra{\pic{nProjector_whiteBlackTrace}\ ^\oplus}
\end{diagram}
\]
where the top arrow is $f\mapsto \Id_P\otimes f$.  The square commutes because $\phi(f) = \e\circ \Tr(\Id_P\otimes f)$ by lemma \ref{epsilonLemma}.  Inverting the topmost (which is a homotopy equivalence by part (2)) and rightmost arrows, we obtain a diagram which commutes up to homotopy
\[
\begin{diagram}
\End(\pic{nProjector}) & \rTo & \End(\pic{nProjector_white-black}^\oplus) & \rTo^{\simeq } & \End(\pic{nProjectorDual})\\
& & \dTo^{\Tr} & & \uTo_{\phi\inv}\\
& & \End(\pic{nProjector_whiteBlackTrace}\ ^\oplus)  & \rTo^{L_\e} & q^n\bra{\pic{nProjector_whiteBlackTrace}\ ^\oplus}
\end{diagram}
\]
where the left-most horizontal arrow is $g\mapsto g\otimes \Id$.  The composition along the top row is precisely $\pi_{P^\vee}$.  One checks that the composition in the other direction sends $g$ to $\phi\inv$ of the map $\e\circ \Tr(g\otimes \Id_{P^\vee}) = \phi(g^\vee)$ (again, see lemma \ref{epsilonLemma}).  In other words, $\pi_{P^\vee}$ is homotopic to $g\mapsto g^\vee$.  This completes the proof.
\end{proof}

\begin{proposition}\label{sheetModuleProp}
Let $Q\in \Ch(\bn{n}^\oplus)$ be a tensor product $Q = Q_1\otimes \dots \otimes Q_r$ with $Q_i\in\{P_n, P_n^\vee\}$ for all $i$.  Any two of the resulting (left or right) actions of $\End(P_n)$ on $Q$ are homotopic.
\end{proposition}
Applying $(\ \ )^\vee$, we have a similar result with $\otimes$ replaced by $\cotimes$.
\begin{proof}
We need only handle the cases (i) $Q = P\otimes P$, (ii) $Q = P\otimes P^\vee$, and (iii) $Q = P^\vee\otimes P^\vee$.

(i) In case $Q = P\otimes P$, part (2) of lemma \ref{sheetModuleLemma} gives that $f\mapsto f\otimes \Id_P$ is homotopic to $f\mapsto \Id_P\otimes \pi_P(f)$, which is in turn homotopic to $f\mapsto \Id_P\otimes f$ by part (3) of the same lemma.  

(ii) The case $Q = P\otimes P^\vee$ is handled similarly.

(iii) $P^\vee\otimes P^\vee = (P\otimes P)^\vee$, so the argument above takes care of this case.
\end{proof}

\begin{corollary}\label{sheetAlgCommCor}
$\End(P_n)$ is homotopy commutative, i.e.~$f\otimes g\mapsto f\circ g$ and $f\otimes g\mapsto (-1)^{|f||g|}g\circ f$ are homotopic as maps $\End(P)^{\otimes 2}\rightarrow \End(P)$.
\end{corollary}
\begin{proof}
Under the isomorphism $\End(P_n)\cong q^n\ket{\pic{nProjector_whiteBlackTrace}\ ^\Pi}$ implied by theorem \ref{dualityIsomorphism}, the left and right regular actions of $\End(P)$ on itself correspond to the maps $f\mapsto \Tr(f\cotimes \Id)$ and  $f\mapsto \Tr(\Id\cotimes f^\vee)$, respectively (this is just a special case of the naturality statement in that theorem).  But proposition \ref{sheetModuleProp} says that these actions coincide up to homotopy.  Thus, $(f\mapsto L_f)$ and $(f\mapsto R_f)$ are homotopic as maps $\End(P_n)\rightarrow \End(\End(P_n))$.  It is an easy check that this implies the statement.
\end{proof}

\subsection{The unknots as algebras}
We give an explicit description of the action of $q^n\ket{\pic{nProjector_trace}\ }$ on $\pic{nProjector}$ .  This section is reminiscent of section \ref{ordinaryDualitySection}.

\begin{definition}
Let $s = \pic{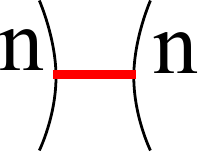}\ $ denote the map $\pic{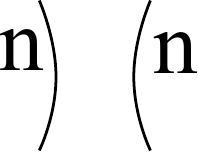}\rightarrow \pic{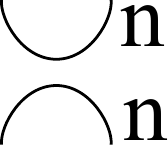} $ of $q$-degree $n$ consisting of $n$ parallel saddle cobordisms.  Let $\eta:\emptyset\rightarrow q^n \pic{nProjector_trace}$ be the map of $q$-degree $-n$ which is $n$ parallel ``cap'' cobordisms $\eta = (\emptyset\rightarrow \underbrace{\pic{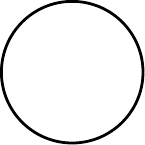}\sqcup \cdots \sqcup \pic{circle}}_n = (\pic{nProjector_trace}\ )^0\hookrightarrow \pic{nProjector_trace}\ )$.
\end{definition}
The saddle induces a chain map $\psi: q^n\ \ket{\pic{nProjector_trace}\ }\rightarrow \End(\pic{nProjector})$ given by
\[
\psi(\zeta) = (\pic{nProjector} \buildrel\zeta\sqcup\Id\over\longrightarrow \pic{nProjector_trace}\  \pic{nProjector}\buildrel\mpic{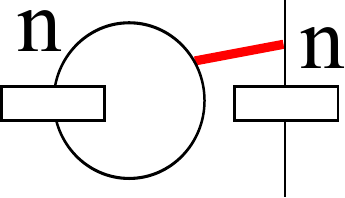}\ \ \ \ \ \over\longrightarrow \pic{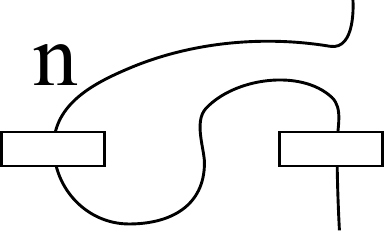}\ \ \ \ \  \simeq \pic{nProjector} ),
\]
where the final map is given by sliding the projectors so they are adjacent and applying a standard equivalence which merges them.  We also have a map $\phi:\End(\pic{nProjector})\rightarrow q^n\ket{\pic{nProjector_trace}\ }$ defined by $\phi(f) = \Tr(f)\circ \eta$.  For the following proposition it is useful to recall the definitions and results of section \ref{ordinaryDualitySection}.  In particular, the diagrams in that section still give a useful way for visualizing the maps $\eta,s,\phi,$ and $\psi$.
\begin{proposition}\label{unknotAction}
The action of $q^n\ket{\pic{nProjector_trace}\ }$ on $\pic{nProjector}$ induced by saddle cobordisms agrees with the action implied by proposition \ref{endRings}, with $\eta$ acting as the unit.  More precisely, let $s,\eta,\phi,$ and $\psi$ be as in the preceding discussion.  Then
\begin{enumerate}
\item $\phi$ is precisely the homotopy equivalence implied by proposition \ref{endRings}.
\item $\psi(\eta)\simeq \Id_P$.
\item $\phi$ and $\psi$ are homotopy inverses.
\end{enumerate}
\end{proposition}
\begin{proof}
(1) Fix $n$, put $P:=P_n$ and $1:=1_n$, and let $\iota:1\rightarrow P$ be the inclusion of the degree zero chain group.  Let $\Phi_{A,B}:\Hom^\bullet_{\bn{n}}(A,B)\rightarrow q^n \ket{B\cotimes A^\vee}$ be the natural isomorphism implied by Theorem \ref{dualityIsomorphism}, and let $\Psi:q^n\ket{\pic{nProjector_whiteBlackTrace}\ ^\Pi}\rightarrow q^n\ket{\pic{nProjector_trace}\ }$ be the standard equivalence, given by $\Psi = \ket{\Tr(\Id_P\cotimes \iota^\vee)}$.  We intend to show that $\Tr(\Id\cotimes \iota^\vee)\circ \Phi_{P,P}(\Id) = \eta$, from which it will follow that the standard equivalence $\Psi\circ \Phi_{P,P}:\End(P)\rightarrow q^n\ket{\Tr(P)}$ sends $f\in\End(P)$ to
\[
\Tr(\Id_P\cotimes \iota^\vee)\circ \Tr(f\cotimes \Id_{P^\vee})\circ  \Phi_{P,P}(\Id)= \Tr(f)\circ \Tr(\Id_P\cotimes \iota^\vee)\circ \Phi_{P,P}(\Id)=\Tr(f)\circ \eta = \phi(f)
\]
which is part (1) of the theorem.  Indeed, consider the following diagram, which commutes by naturality of $\Phi_{A,B}$.
\[
\begin{diagram}
\End(1) & \rTo^{L_\iota} & \Hom(1, P) & \lTo^{R_\iota} & \End(P)\\
\dTo^{\Phi_{1,1}} & & \dTo^{\Phi_{1,P}} & & \dTo^{\Phi_{P,P}}\\
q^n\ket{\pic{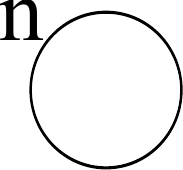}\ } &\rTo^{\ket{\Tr(\iota)}} & q^n\ket{\pic{nProjector_trace}\ } & \lTo^{\ket{\Tr(\Id\cotimes \iota^\vee)}} & q^n\ket{\pic{nProjector_whiteBlackTrace}\ ^\Pi}
\end{diagram}
\]
We want to see that the image of $\Id_P$ (in the top right corner) is $\eta$ (in the chain complex appearing in the middle of the bottom row).  By commutativity of the square on the right, this image is equal to $\Phi_{1,P}(\iota)$.  By commutativity of the square on the left, this is in turn equal to $\Tr(\iota)\circ \Phi_{1,1}(\Id_1)$.  By ordinary duality (proposition \ref{ordinaryDualityIso}) we have $\Phi_{1,1}(\Id) =\eta_1$, where $\eta_1:\emptyset\rightarrow \pic{nStrands_unknot}\ $ is $n$ disjoint cup cobordisms.  It follows that
\[
\Tr(\Id\cotimes \iota^\vee)\circ \Phi_{P,P}(\Id)= \Tr(\iota)\circ \eta_1 = \eta.
\]
The last equality here is a simple observation which follows immediately from the definitions.  This proves (1).

(2) We have a commutative diagram
\[
\begin{diagram}
\pic{nProjector} & \rTo^{\eta_1\sqcup \Id} & \pic{nStrands_unknot}\ \pic{nProjector} & \rTo^{\Tr(\iota)\sqcup \Id} & \pic{nProjector_trace}\ \pic{nProjector} & & \\ 
& & \dTo^{\mpic{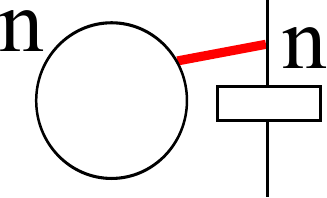}\ \ \ \ } & & \dTo^{\mpic{unknot-arc_saddle}\ \ \ } & & \\
& & \pic{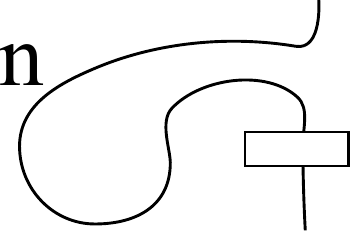}\ \ \ \ \ & \rTo^{\iota\otimes \Id} & \pic{nProjector_traceAndProjectorMerged}\ \ \ \ \  & \rTo^{\simeq} & \pic{nProjector}
\end{diagram}
\]
One composition is $\psi(\Tr(\iota)\circ \eta_1) = \psi(\eta)$, and the other is homotopic to the identity since $\eta_1$ followed by $n$ parallel saddle cobordisms gives the identity on $n$ strands (see section \ref{ordinaryDualitySection}).  This proves (2)

(3)  The maps $s,\eta$, $\phi$, and $\psi$ can be represented schematically as in section \ref{ordinaryDualitySection}.  The proof that $\psi$ and $\phi$ are inverses follows the same lines, using (2).
\end{proof}
There are similar results describing the actions of $\ket{\ \pic{thetaNet}\ \ }$ and $\ket{ \ \pic{pillNet}\ \ \ \ }$ on $\pic{3network}\ $ and $\ \pic{iNet_i}\ \ $, respectively.
\subsection{The 2-colored unknot}
We can use proposition \ref{unknotAction} to describe explicitly the algebra structure of the unknots in terms of the saddle maps and the equivalence $\phi:\pic{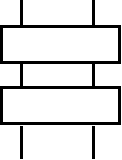}\simeq \pic{2projector}$.  The homology of $\End(P_2^\vee)$ was computed in \cite{CHK11}, and explicit generators were found.  Here we make an explicit connection  with the trace of $P_2$.  We state without proof the following:
\begin{proposition}  The saddle map makes $q^2\ket{\pic{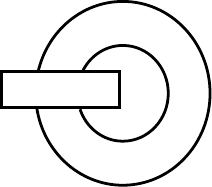}\ }$ into a dg algebra isomorphic to $\Z[\a,b_1,b_2,v]/(b_i^2=\a, b_1v=b_2v)$ with differential generated by $d(b_i)=0$, $d(v)=b_1+b_2$ under the Leibniz rule.  Here, the degrees of the generators are $\deg(b_i)=(0,2)$, $\deg(\a)=(0,4)$, and $\deg(v)=(-1,2)$.  Moreover, the action of this algebra on $\pic{2projector}$ is given by $b_1\mapsto \pic{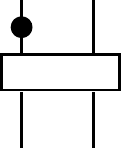}$, $b_2=\pic{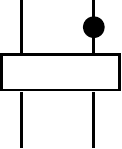}$, and 
\begin{equation}
\label{v}
v\mapsto
\left(\begin{diagram}
& \cdots & \rTo^{\differenceOfDots} & q^5 \turnback & \rTo^{\sumOfDots} & q^3 \turnback & \rTo^{\differenceOfDots} & q \turnback & \rTo^{\isaddle} & \underline\straightthrough\\
\ldTo^\Id & & \ldTo^\Id & & \ldTo^\Id & & \ldTo^\Id & & \ldTo^{\mpic{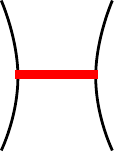}}\\
& \cdots & \rTo_{\differenceOfDots} & q^5 \turnback & \rTo_{\sumOfDots} & q^3 \turnback & \rTo_{\differenceOfDots} & q \turnback & \rTo_{\isaddle} & \underline\straightthrough
\end{diagram}\right)\qed
\end{equation}
\end{proposition}

We remark that even though $q^2\ket{\pic{2projector_trace}\ }$ is an honestly commutative dg algebra (not graded commutative), it is graded commutative up to homotopy.  In particular, any odd dimensional class squares to zero.  Since $v$ does not represent a class in homology, even powers of $v$ can still represent nontrivial homology classes.  Finally, we remark that this formula for the unknot agrees (setting $\a=0$) with the formula given in \cite{GOR12}, namely
\[
q^2H(\ket{\pic{2projector_trace}\ })\cong H(\Z[x_i,y_i:i=0,1])
\]
with differential $d(x_i)=0$, $d(y_0)=x_0^2$, $d(y_1)=x_0x_1+x_1x_0 = 2x_0x_1$.  The correspondence in homology is generated by $[x_0]=[b_1]$, $[x_1]=[v^2]$, and $[x_0y_1-2x_1y_0]=[b_1v^3]$.
\section{Appendix: homotopy lemmas}\label{homotopySection}
\subsection{Contracting summands}
\begin{lemma}\label{contractibleComplexesLemma}
Suppose $(B,d)$ is a contractible chain complex, i.e. there is $h:B\rightarrow B$ such that $\Id_B = d\circ h +h\circ d$.  Then
\begin{enumerate}
\item $d$ and $h^2$ commute.
\item $e = d\circ h$ and $f: h\circ d$ are orthogonal idempotents.
\item $h':=h\circ d\circ h$ is another nulhomotopy, and satisfies $(h')^2=0$.
\end{enumerate}
\end{lemma}
\begin{proof}
Straightforward.
\end{proof}

\begin{proposition}[Gaussian elimination]\label{Gauss}
Let $\mathscr{A}$ be a $\Z$-linear category, and suppose we have a chain complex $C = (A\oplus B, D:=\matrix{{}_Ad_A & {}_A d_B \\ {}_B d_A & {}_B d_B})$.  Suppose $(B,{}_Bd_B)$ is a contractible chain complex.  Then $C$ is homotopy equivalent to a chain complex of the form $ A'= (A, {}_A d_A - {}_A d_B\circ h \circ {}_B d_A)$.
\end{proposition}

\begin{proof}
By hypotheses there is some nulhomotopy $h:B\rightarrow B$, and by lemma \ref{contractibleComplexesLemma} we may assume $h^2=0$.  The relevant maps are defined in the following diagram:
\[
\begin{tikzpicture}
\matrix(m)[matrix of math nodes,
row sep=2.6em, column sep=10em,
text height=1.5ex, text depth=0.25ex]
{\AplusB & A'\\};
\path[->] 
(m-1-1)	edge [loop left, looseness = 0, min distance=20 mm]
		node[below=4mm] {$H=\matrix{0 & h \\ 0 & 0}$}									(m-1-1)
		edge [bend left = 15] node[above] {$r = \matrix{\Id&-{}_A d_B\circ h}$}		(m-1-2)
(m-1-2)	edge [bend left = 15] node[below] {$i = \matrix{\Id\\-h\circ {}_B d_A}$}		(m-1-1);
\end{tikzpicture}
\]
It is straightforward to check that (1) $r$ and $i$ are chain maps, (2) $r\circ i = \Id_A$, (3) $\Id_{A\oplus B}-i\circ r = D\circ H+ H\circ D$, (3) $r\circ H =0$, (4) $H\circ i =0$.  In particular $r$ and $i$ are inverse homotopy equivalences (actually we say that $r$ is a strong deformation retract). 
\end{proof}

In particular if $B$ is a ``subcomplex'' or ``quotient complex'' (i.e.\ ${}_A d_B =0$ or ${}_B d_A =0$, respectively), then $C\simeq (A, {}_A d_A)$.  That is to say, the differential on $A$ is unaffected by contracting $B$ in this case.  More generally, suppose (1) $X$ is some partially ordered set, (2) $A_x^\bullet$ are graded objects indexed by $x\in X$, and (3) we have maps ${}_y d_x:A_x\rightarrow A_y$ of degree 1, $x\leq y$, such that $\sum_{x\leq y}{}_xd_y$ is a differential on $\bigoplus_{x\in X} A_x$.  In particular, each $(A_x, {}_xd_x)$ is a chain complex.

Suppose $Y\subset X$ is a finite subset such that $A_y\simeq 0$ for $y\in Y$.  By iterating proposition \ref{Gauss}, we may contract each $A_y$, $y\in Y$ obtaining 
\[
\Big(\bigoplus_{x\in X} A_x, \sum_{\frac{x,y\in X}{x\leq y}} {}_y d_x\Big) \simeq \Big(\bigoplus_{x\in X\setminus Y} A_x, \sum_{x, y\in X\setminus Y\atop x\leq y}{}_y d'_x\Big)
\]
where ${}_y d'_x - {}_y d_x$ is a sum over paths $x < z_1 < \dots < z_r < y$ with $z_i\in Y$ of maps of the form $\pm {}_yd_{z_r}\circ h\circ \dots \circ h\circ {}_{z_1}d_x$.  We are interested to know when we can contract infinitely many summands in this manner.  It suffices that for each $x\in X$ the set $\{x' \in X\:|\: x<x'\}$ be finite, but this is not necessary.  For our present purposes it will suffice to consider the case of bicomplexes ($X=\Z$ and ${}_j d_i=0$ unless $0\leq j-i\leq 1$), which we describe next.

\subsection{Bicomplexes with contractible columns} \label{bicomplexesSection}

\begin{definition}
Let $\mathscr{A}$ be an additive category.  A \emph{bicomplex over $\mathscr{A}$} is a triple $(A^{\bullet\bullet}, \d, \d')$, where $A^{\bullet\bullet} = (A^{ij})_{i,j\in\Z}$ is a bigraded object over $\mathscr{A}$ and $\d,\d':A^{\bullet \bullet}\rightarrow A^{\bullet \bullet}$ are bihomogeneous maps of degree $(1,0)$, respectively $(0,1)$ such that $(\d+\d')^2=0$.  
\end{definition}

By expanding the equation $(\d+\d')^2=0$ into its bihomogeneous components we obtain (1) $\d^2=0$, (2) $(\d')^2=0$, and (3) $\d\circ \d'+\d'\circ \d =0$.  In particular the \emph{columns} $(A^{i,\bullet},\d')$ and \emph{rows} $(A^{\bullet, j},\d)$ are chain complexes.

\begin{definition}
Let $(A^{\bullet\bullet},\d,\d')$ be a bicomplex.  We can form two chain complexes over $\mathscr{A}$ which we call \emph{total complexes of type I, respectively II}.  These occasionally involve forming countable direct sums or products, and so are not always defined.
\begin{itemize}
\item[(I)] Let $\Tot^\oplus(A)$ denote the chain complex $\Tot^\oplus(A)^k = \bigoplus_{i+j=k}A^{ij}$ with differential $\d+\d'$, when defined.
\item[(II)] Let $\Tot^\Pi(A)$ denote the chain complex $\Tot^\Pi(A)^k = \prod_{i+j=k}A^{ij}$ with differential $\d+\d'$, when defined.
\end{itemize}
\end{definition}

Note that If $A^{\bullet\bullet}$ is supported in quadrant I and/or III then the sum $\bigoplus_{i+j=k}A^{ij}$ and the product $\prod_{i+j=k}A^{ij}$ are finite for each $k$.  So in this case  $\mathscr{A}$ need only contain finite sums (equivalently products)  to form $\Tot^\oplus(A)$ and $\Tot^\Pi(A)$.  In this case $\Tot^\oplus(A)\cong \Tot^\Pi(A)$.

\begin{proposition}\label{contractibleBicomplexes}
Suppose $(A^{\bullet\bullet}, \d,\d')$ is a bicomplex with contractible columns.
\begin{enumerate}
\item If $A^{ij}=0$ for $(i,j)$ in the fourth quadrant, then $\Tot^\oplus(A)\simeq 0$.
\item  Dually, if $A^{ij}=0$ for $(i,j)$ in the second quadrant, then $\Tot^\Pi(A)\simeq 0$.
\end{enumerate}
\end{proposition}
\begin{proof}
We first prove (1) when $A^{ij}$ is supported in the left half-plane $i\leq 0$ and (2) when $A^{ij}$ when $A^{ij}$ is supported in the right half-plane $i\geq 0$.  The general cases will involve patching these two cases together.

\textbf{Step 1.}  Suppose $A^{ij}=0$ for $i>0$.  Put $C_i = (A^{i,\bullet}, \d)$, so that $C_i =0$ for $i>0$.  By hypotheses we have $h:C_i\rightarrow C_i$ of degree $-1$ such that $\Id_{C_i} = \d\circ h+h\circ \d$.  Define $H:\Tot^\oplus(A)\rightarrow \Tot^\oplus$ by
\[
H = \sum_{m \geq 0} (-1)^m h \circ (\d'\circ h)^m
\]
This map is defined since for fixed $i$, the restriction $h \circ (\d'\circ h)^m|_{C_i}$ vanishes for $m>j$.  Hence $H|_{C_i} = \sum_{0\leq m\leq i} ((-1)^m h \circ (\d'\circ h)^m)|_{C_i}$ is defined.  $H$ is now defined by the universal property of direct sums.

One can check formally that $(\d+\d')\circ H+H\circ (\d+\d') = \Id_{\Tot^\oplus(A)}$.  We leave this to the reader.  That is to say, $\Tot^\oplus(A)\simeq 0$ in this case.

\textbf{Step 2.}  Suppose $A^{ij}=0$ for $i<0$.  A slight modification of the previous argument implies that $\Tot^\Pi(A)\simeq 0$ in this case.

\textbf{Step 3.}  Suppose $A^{ij}=0$ for $i>0,j<0$.  Let $B$ and $C$ be the restricted bicomplexes:
\[
B^{ij} =
\begin{cases}
A^{ij} & \text{ for $i\leq 0$}\\ 
0 & \text{ otherwise}
\end{cases}
\hskip.5 in
C^{ij}=
\begin{cases}
A^{ij} & \text{ for $i> 0$}\\ 
0 & \text{ otherwise}
\end{cases}.
\]
Then $B$ and $C$ are bicomplexes with contractible columns.  $B$ is concentrated in the left half-plane, so previous arguments imply $\Tot^\oplus(B)\simeq 0$.  $C$ is concentrated in the quadrant I, so $\Tot^\oplus(C)\cong \Tot^\Pi(C)$, which is contractible by previous arguments.  Further, $\Tot^\oplus(A)$ is a mapping cone $\Tot^\oplus(A)\cong (\Tot^\oplus(B)\rightarrow \Tot^\oplus(C))$, which is contractible by two applications of proposition \ref{Gauss}.

\textbf{Step 4.}  Suppose $A^{ij}=0$ for $i>0,j<0$.  A similar argument to above implies that $\Tot^\Pi(A)\simeq 0$.  This completes the proof.
\end{proof}

\end{document}